\documentclass[a4paper, 10pt, DIV=10, headinclude=false, footinclude=false, leqno]{article}
\usepackage[T1]{fontenc}
\usepackage[utf8]{inputenc}
\usepackage[english]{babel}

\usepackage{amsmath,amssymb,amsthm}
\allowdisplaybreaks
\usepackage[bookmarks,
            raiselinks,
            pageanchor,
            hyperindex,
            colorlinks,
            citecolor=black,
            linkcolor=black,
            urlcolor=black,
            filecolor=black,
            menucolor=black]{hyperref}

\usepackage[capitalize,nameinlink]{cleveref}

\crefname{section}{section}{sections}
\crefname{subsection}{subsection}{subsections}
\Crefname{section}{Section}{Sections}
\Crefname{subsection}{Subsection}{Subsections}

\Crefname{figure}{Figure}{Figures}

\crefformat{equation}{\textup{#2(#1)#3}}
\crefrangeformat{equation}{\textup{#3(#1)#4--#5(#2)#6}}
\crefmultiformat{equation}{\textup{#2(#1)#3}}{ and \textup{#2(#1)#3}}
{, \textup{#2(#1)#3}}{, and \textup{#2(#1)#3}}
\crefrangemultiformat{equation}{\textup{#3(#1)#4--#5(#2)#6}}%
{ and \textup{#3(#1)#4--#5(#2)#6}}{, \textup{#3(#1)#4--#5(#2)#6}}{, and \textup{#3(#1)#4--#5(#2)#6}}

\Crefformat{equation}{#2Equation~\textup{(#1)}#3}
\Crefrangeformat{equation}{Equations~\textup{#3(#1)#4--#5(#2)#6}}
\Crefmultiformat{equation}{Equations~\textup{#2(#1)#3}}{ and \textup{#2(#1)#3}}
{, \textup{#2(#1)#3}}{, and \textup{#2(#1)#3}}
\Crefrangemultiformat{equation}{Equations~\textup{#3(#1)#4--#5(#2)#6}}%
{ and \textup{#3(#1)#4--#5(#2)#6}}{, \textup{#3(#1)#4--#5(#2)#6}}{, and \textup{#3(#1)#4--#5(#2)#6}}

\crefdefaultlabelformat{#2\textup{#1}#3}

\usepackage{enumitem}
\usepackage{caption}
\usepackage{wrapfig}
\usepackage{sidecap} 
\usepackage{csquotes}
\usepackage{url}
\usepackage{mathtools}
\usepackage{xfrac}
\usepackage[
   algo2e,
   lined,
   boxed,
  commentsnumbered,
  linesnumbered,
  ruled
 ]{algorithm2e}
\usepackage{xparse}

\usepackage{pgfplots}
\usepackage{pgfkeys}
\def\addlegendimage{\csname pgfplots@addlegendimage\endcsname}

\usepackage{todonotes}
\usepackage{color}
\usepackage{setspace}
\definecolor{sixclassRdYlBu1}{rgb}{0.84,0.19,0.15}
\definecolor{sixclassRdYlBu2}{rgb}{0.99,0.55,0.35}
\definecolor{sixclassRdYlBu3}{rgb}{1.0,0.88,0.56}
\definecolor{sixclassRdYlBu4}{rgb}{0.88,0.95,0.97}
\definecolor{sixclassRdYlBu5}{rgb}{0.57,0.75,0.86}
\definecolor{sixclassRdYlBu6}{rgb}{0.27,0.46,0.71}

\usepackage{subcaption} 
\theoremstyle{plain}
\newtheorem{lemma}{Lemma}[section]
\newtheorem{theorem}[lemma]{Theorem}

\newtheorem{proposition}[lemma]{Proposition}

\newcommand{\hS}[1]{\hspace{#1pt}}

\DeclareMathOperator{\integralend}{d}
\newcommand{\intend}[1]{\integralend \hS{-1.25} #1}
\newcommand{\dx}{\intend{x}}

\DeclareMathOperator*{\argmin}{argmin}
\DeclareMathOperator*{\argmax}{arg\,max}
\DeclareMathOperator{\supp}{supp}

\DeclareMathOperator{\Span}{span}

\newcommand{\Params}{\mathcal{P}}
\newcommand{\ParamsTrain}{\mathcal{P}_\textnormal{train}}
\newcommand{\ParamsVer}{\mathcal{P}_\textnormal{val}}

\newcommand{\dimVH}{{N_H}}
\newcommand{\dimVh}{{N_h}}

\newcommand{\grid}{\mathcal{T}_h}

\newcommand{\Gridh}{\mathcal{T}_H}

\newcommand{\Kmu}{\mathbb{K}_\mu}
\newcommand{\Kmurb}{\mathbb{K}_\mu^{rb}}
\newcommand{\Ktmu}{\mathbb{K}_{T,\mu}}
\newcommand{\Ktmurb}{\mathbb{K}_{T,\mu}^{rb}}
\newcommand{\Ktqo}{\mathbb{K}_{T,q}^{0}}
\newcommand{\Ktqrb}{\mathbb{K}_{T,q}^{rb}}
\newcommand{\KT}{J_T}

\newcommand{\uhkm}[1][\mu]{u_{H,k,#1}}
\newcommand{\uhkmcoeff}{\underline{u}_{H,k,\mu}}

\newcommand{\uhkmsm}{u_{H,k,\mu}^{\text{ms}}}
\newcommand{\uhkmsmcoeff}{\underline{u}_{H,k,\mu}^{\text{ms}}}

\newcommand{\vf}{v^{\text{f}}}
\newcommand{\vft}[1][T]{v^\text{f}_{#1}}
\newcommand{\uft}[1][T]{u^\text{f}_{#1}}
\newcommand{\ufti}[1]{u^\text{f}_{T_{#1}}}
\newcommand{\vfti}[1]{v^\text{f}_{T_{#1}}}

\newcommand{\IH}{\mathcal{I}_H}
\newcommand{\CI}{C_{\mathcal{I}_H}}
\newcommand{\CO}{C_{k,\text{ovl}}}
\newcommand{\CPG}{\gamma_k}
\newcommand{\QQ}{{\mathcal{Q}}}
\newcommand{\QQm}{{\QQ_{\mu}}}
\newcommand{\QQkm}{{\QQ_{k,\mu}}}
\NewDocumentCommand{\QQktm}{O{T} O{\mu}}{\QQ^{#1}_{k,#2}}
\NewDocumentCommand{\QQktmrb}{O{T} O{\mu}}{\QQ^{#1,rb}_{k,#2}}
\NewDocumentCommand{\QQktmts}{O{T} O{\mu}}{\QQ^{#1,TS}_{k,#2}}
\newcommand{\Vfh}{V^{\text{f}}_{h}}
\newcommand{\Vfhkt}[1][T]{V^{\text{f}}_{h,k,#1}}
\newcommand{\Vfrbkt}[1][T]{V^{\text{f},rb}_{k,#1}}
\newcommand{\Wfrbkt}[1][T]{W^{\text{f},rb}_{k,#1}}

\newcommand{\Tlast}{T_{\abs{\mathcal{T}_H}}}
\newcommand{\Vts}{\mathfrak{V}}
\newcommand{\Wtsrb}{\mathfrak{W}^{rb}}
\newcommand{\Vtsrb}{\mathfrak{V}^{rb}}
\newcommand{\uts}{\mathfrak{u}}
\newcommand{\utsm}{\mathfrak{u}_{\mu}}
\newcommand{\utsmrb}{\mathfrak{u}_{\mu}^{rb}}
\newcommand{\vts}{\mathfrak{v}}
\newcommand{\Btsm}{\mathfrak{B}_{\mu}}
\newcommand{\Btsq}[1][q]{\mathfrak{B}_{q}}
\newcommand{\Fts}{\mathfrak{F}}
\newcommand{\sumT}{\sum_{T\in\Gridh}}
\newcommand{\Qa}{Q}

\newcommand{\dimVfrbkt}{N_T}
\newcommand{\rbt}[2][T]{\varphi_{#1,#2}}
\newcommand{\rbts}[1]{\mathfrak{b}_{#1}}
\newcommand{\rbtscoeff}[1]{\underline{\mathfrak{b}}_{#1}}
\newcommand{\rbestt}[2][T]{\psi_{#1,#2}}
\newcommand{\rbestts}[1]{\mathfrak{c}_{#1}}
\newcommand{\rbesttscoeff}[1]{\underline{\mathfrak{c}}_{#1}}

\DeclarePairedDelimiterX\mengenA[1]{\lbrace}{\rbrace}{#1}
\DeclarePairedDelimiterX\mengenB[2]{\lbrace}{\rbrace}{#1\, \delimsize\vert \, #2}

\makeatletter
\newcommand{\set}[2][\relax]{
	\ifx#1\relax \ensuremath{
		\mengenA*{#2}}
	\else \ensuremath{%
		\mengenB*{#1}{#2}}
	\fi}
\makeatother

\DeclareRobustCommand{\minwidthbox}[2]{%
	\ifmmode
	\expandafter\mathmakebox
	\else
	\expandafter\makebox
	\fi
	[\ifdim#2<\width\width\else#2\fi]{#1}%
}

\DeclarePairedDelimiter{\abs}{|}{|}
\DeclarePairedDelimiterX\skal[2]{(}{)}{#1\,,\,#2}
\DeclarePairedDelimiter{\norm}{\lVert}{\rVert}

\DeclareFontFamily{U}{matha}{\hyphenchar\font45}
\DeclareFontShape{U}{matha}{m}{n}{
      <5> <6> <7> <8> <9> <10> gen * matha
      <10.95> matha10 <12> <14.4> <17.28> <20.74> <24.88> matha12
      }{}
\DeclareSymbolFont{matha}{U}{matha}{m}{n}
\DeclareFontSubstitution{U}{matha}{m}{n}

\DeclareFontFamily{U}{mathx}{\hyphenchar\font45}
\DeclareFontShape{U}{mathx}{m}{n}{
      <5> <6> <7> <8> <9> <10>
      <10.95> <12> <14.4> <17.28> <20.74> <24.88>
      mathx10
      }{}
\DeclareSymbolFont{mathx}{U}{mathx}{m}{n}
\DeclareFontSubstitution{U}{mathx}{m}{n}

\DeclareMathDelimiter{\vvvert}{0}{matha}{"7E}{mathx}{"17}

\DeclarePairedDelimiterXPP{\snorm}[1]{}{\lVert}{\rVert}{_{1}}{\ifblank{#1}{\:\cdot\:}{#1}}
\DeclarePairedDelimiterXPP{\anorm}[1]{}{\lVert}{\rVert}{_{a,\mu}}{\ifblank{#1}{\:\cdot\:}{#1}}
\DeclarePairedDelimiterXPP{\Snorm}[1]{}{\vvvert}{\vvvert}{_{1}}{\ifblank{#1}{\:\cdot\:}{#1}}
\DeclarePairedDelimiterXPP{\Anorm}[1]{}{\vvvert}{\vvvert}{_{a,\mu}}{\ifblank{#1}{\:\cdot\:}{#1}}
\DeclarePairedDelimiterXPP{\SMnorm}[1]{}{\vvvert}{\vvvert}{_{1,\mu}}{\ifblank{#1}{\:\cdot\:}{#1}}

\makeatletter
\newcommand{\labeltext}[2]{%
  \@bsphack
  \csname phantomsection\endcsname 
  \def\@currentlabel{#1}{\label{#2}}%
  \@esphack
}
\makeatother

\allowdisplaybreaks

\DeclareOldFontCommand{\rm}{\normalfont\rmfamily}{\mathrm}

\title{An Online Efficient Two-Scale Reduced Basis Approach for the Localized Orthogonal Decomposition
	\footnote{{\textbf{Funding:}} The authors acknowledge funding by the Deutsche Forschungsgemeinschaft under Germany’s Excellence Strategy EXC 2044 390685587, Mathematics M\"unster: Dynamics -- Geometry -- Structure.
		The first author also acknowledges funding by the DFG contract OH 98/11-1.
}}

\author{Tim Keil$^\dagger$, Stephan Rave\thanks{Mathematics Münster, Westfälische Wilhelms-Universität Münster, Einsteinstr. 62, D-48149 M\"unster, \url{{tim.keil,stephanrave}@uni-muenster.de}}}

\usepackage{amsopn}

\usepackage{geometry}
\ifpdf
\hypersetup{
	pdftitle={An Online Efficient Two-Scale Reduced Basis Approach for the Localized Orthogonal Decomposition},
	pdfauthor={T. Keil, and S. Rave}
}
\fi

\begin{document}

\maketitle

\begin{abstract}
	We are concerned with employing Model Order Reduction (MOR) to efficiently solve parameterized multiscale problems
    using the Localized Orthogonal Decomposition (LOD) multiscale method.
	Like many multiscale methods, the LOD follows the idea of separating the problem into localized fine-scale
	subproblems and an effective coarse-scale system derived from the solutions of the local problems.
	While the Reduced Basis (RB) method has already been used to speed up the solution
	of the fine-scale problems, the resulting coarse system remained untouched, thus limiting the achievable speed
    up.
	In this work we address this issue by applying the RB methodology to a new two-scale formulation of
	the LOD.
    By reducing the entire two-scale system, this two-scale Reduced Basis LOD (TSRBLOD) approach, 
    yields reduced order models that are completely
	independent from the size of the coarse mesh of the multiscale approach, allowing an
	efficient approximation of the solutions of parameterized multiscale problems even for very
	large domains. 
    A rigorous and efficient a posteriori estimator bounds the model reduction error,
    taking into account the approximation error for both the local fine-scale problems and
	the global coarse-scale system.

\par\vskip\baselineskip\noindent

\par\vskip\baselineskip\noindent
\end{abstract}

\section{Introduction}
\label{sec:introduction} 
The numerical approximation of partial differential equations that exhibit multiscale structures has many applications in the natural sciences.
Typical examples are, for instance, the simulation of composite materials or of flow in porous media.
Due to the different spatial or temporal scales at which the physics in such problems are modeled,
a direct solution with standard methods such as finite elements is often prohibitively expensive, as very fine meshes over large computational domains are required to resolve all relevant features.

Numerical multiscale methods are designed to overcome this issue by resolving the micro-structures only locally to solve small cell problems that yield the fine-scale information required to build an effective coarse-scale system, which then can be solved with modest computational effort.
Some of these multiscale methods, such as the Heterogeneous Multiscale Method (HMM)~\cite{EEnq},
are based on ideas from mathematical homogenization and aim at computing effective coefficients for an appropriate coarse-scale equation. 
Other approaches, such as the Multiscale Finite Element Method (MsFEM)~\cite{MsFEM,EH09} or the Generalized Finite Element Method (GFEM)~\cite{LiptonBabuska}, instead construct coarse-scale elements that incorporate the local fine-scale features of the solution and then approximate the solution in the space spanned by these multiscale elements.

The Localized Orthogonal Decomposition (LOD)~\cite{MP14}, which itself is based on the Variational Multiscale Method (VMM)~\cite{HFMQ1998}, instead makes use of a splitting of the full fine-scale approximation space into a negligible fine-scale component and an orthogonal multiscale space in which the solution is sought.
This multiscale space is then approximated by computing localized approximations of the orthogonal projection into the finescale space (corrector problems).
For an extensive overview of the LOD, we refer to~\cite{LODbook} and the references therein.
In this work we consider a Petrov--Galerkin formulation of the LOD~\cite{elf},
which has lower storage and communication requirements for the computed fine-scale data in comparison to the original Galerkin formulation.

Many problems modeled with partial differential equations include parameters which need to be varied, e.g., to find an optimal choice of parameter values w.r.t.\ some cost functional.
Due to the large costs of repeatedly solving the problem for different parameters in such workflows, Model Order Reduction (MOR) techniques are often used to replace the model at hand by a quickly evaluable surrogate. 
In this work we consider Reduced Basis (RB) methods, which build such a Reduced Order Model (ROM) by projecting the model equations onto a problem adapted subspace spanned by solutions of the Full Order Model (FOM) for appropriately selected \emph{snapshot} parameters.
For a detailed overview and discussion on RB methods, we refer to the monographs~\cite{HesthavenRozzaEtAl2016,QuarteroniManzoniEtAl2016} and the tutorial introduction~\cite{HAA2017}.

For large, parameterized multiscale problems, RB methods need to be combined with numerical multiscale methods, as otherwise the solution of the FOM for a single snapshot parameter might already be computationally infeasible. 
One possible approach is to speed up the solution of the individual cell problems using RB techniques, as is done in~\cite{Bo08}
in the context of numerical homogenization, in~\cite{AB12,AB13,AB14_2,AB14,A19} for the HMM, in~\cite{efendiev2013generalized,HesthavenZhangEtAl2015,Nguyen2008} for the MsFEM, or in~\cite{RBLOD} for the LOD\@.
This approach is applicable both for problems, where the fine-scale data variation over the domain is parameterized, resulting in a single ROM for all cell problems~\cite{AB12,AB13,AB14_2,AB14,A19,Bo08,HesthavenZhangEtAl2015}, or for general parameterized problems, where for each individual cell problem a dedicated ROM is built~\cite{RBLOD, Nguyen2008}.
Some recent works have considered the case where the multiscale data varies only in some cells, e.g., caused by perturbations, by a sequence of modification over time or by other parameterizations, such that individual fine-scale solutions can be reused~\cite{HKM20,HM19,MV20} or a ROM can be built in case of parameterized perturbations~\cite{MV21}.

All these approaches have in common that, while the constructed surrogates are independent of the resolution of the fine-scale mesh, the effort for their evaluation still scales with the size of the used coarse mesh.
Hence, if the coarse mesh itself is large, the repeated assembly and solution of the coarse system still can become a computational bottleneck.
Further, the influence of the approximation error of the cell problem ROMs on the coarse system is not rigorously controlled.

In this work we introduce a new two-scale reduced basis method for the LOD (TSRBLOD) that takes both the local corrector
problems as well as the coarse-scale problem into account to produce a single small-size ROM that no longer requires explicit solutions of local subproblems in the online phase.
The model is constructed based on a new two-scale formulation of the LOD as a single variational problem, inspired by
the theory of two-scale convergence in mathematical homogenization~\cite{allaire}.
As the computation of solution snapshots for this model would still be computationally expensive, we combine this
formulation with a preceding reduction of the fine-scale corrector problems similar to the RBLOD approach in~\cite{RBLOD}.
Rigorous and efficient a posteriori bounds derived from the two-scale formulation control the error over this entire
two-stage reduction process.

While our two-scale formulation of the LOD is new, an online-efficient RB-ROM for locally periodic homogenization problems was developed for the HMM in~\cite{OS12,schaefer13,OhlbergerSchaeferEtAl2018} based on the two-scale formulation in~\cite{Ohlberger2005}.
Further, we mention localized MOR techniques~\cite{BuhrIapichinoEtAl2020}, which can also be employed to obain online-efficient ROMs for large multiscale problems.
Whereas numerical multiscale methods build on the idea of scale separation and computing effective global systems from
local information, these methods are usually based on the idea of decomposing the computational domain and deriving a globally coupled ROM
from local ROMs associated with the domain decomposition. 

This paper is organized as follows. In \cref{sec:problem} we introduce the considered model problem, after which we biefely review the LOD method in \cref{sec:multiscale}.
\Cref{sec:two_scale} is devoted to our two-scale formulation of the LOD.
In particular, we prove inf-sup stability of the introduced two-scale bilinar form, from which we derive error bounds for the two-scale system.
In \cref{sec:MOR} we detail our new two-stage reduction approach for the LOD based on this two-scale formulation.
The numerical experiments in \cref{sec:experiments} show the efficiency of the method, in particular for large problems.
We end with some concluding remarks in \cref{sec:conclusion}.

\section{Problem formulation}\label{sec:problem}
Let $\Params \subset \mathbb{R}^p, p \in \mathbb{N}$ be a parameter space and $\Omega \in \mathbb{R}^d$ a bounded Lipschitz domain. 
We consider the following prototypical parameterized elliptic partial differential equation: for a fixed parameter $\mu
\in \Params$,
find $u_{\mu}$ such that
\begin{equation}
\begin{aligned}
\label{eq:problem_classic}
- \nabla \cdot A_{\mu}(x)  \nabla u_{\mu}(x) &= f(x), \qquad \text{ in }\Omega, \\
u_{\mu}(x) & = 0, \quad \;\; \quad  x \in \partial \Omega.
\end{aligned} 
\end{equation}
We assume that the parameter-dependent coefficient field $A_{\mu}$ has a multiscale structure that renders a direct
solution of~\eqref{eq:problem_classic} using, e.g., finite elements computationally infeasible due to the
high mesh resolution required to resolve all features of $A_\mu$.

Further, we assume $A_{\mu} \in L^{\infty}(\Omega, \mathbb{R}^{d \times d})$ to be symmetric and uniformly elliptic w.r.t.\
$x$ and $\mu$, i.e., there exist $0 < \alpha \leq \beta < \infty$ such that 
\begin{equation}\label{eq:ellipticity}
    \forall\mu\in\Params\ \forall x\in \Omega\ \forall \xi \in \mathbb{R}^d: \ \alpha |\xi|^2 \leq  \xi^T \cdot A_{\mu}(x) \cdot \xi \leq \beta |\xi|^2,
\end{equation}
and we let $\kappa:=\beta/\alpha$ be the maximum contrast of $A_{\mu}$.
Moreover, let $f \in L^2(\Omega)$. 

We consider the weak formulation of \eqref{eq:problem_classic}:
for $\mu \in \Params$, we seek $u_{\mu} \in V:=H^1_0(\Omega)$ such that 
\begin{equation}
\label{eq:problem_weak}
a_{\mu}(u_\mu,v) =  F(v) \qquad \forall v \in V,
\end{equation}
where
\begin{equation*}
    a_{\mu}(u,v) := \int_{\Omega}^{} \left(  A_{\mu}(x)  \nabla u(x) \right) \cdot  \nabla v(x)\dx \qquad
    \text{ and } \qquad F(v) := \int_\Omega f(x)v(x)\dx.
\end{equation*} 
Due to our assumptions, $a_{\mu}$ is a continuous, coercive bilinear form on $V$ and $F \in V^\prime$, s.t.~\eqref{eq:problem_weak}
admits a unique solution by the Lax-Milgram theorem.

As a final assumption, which will be required in \cref{sec:MOR} to obtain an online-efficient
reduced order model, we assume parameter-separability of $A_\mu$, i.e.,
we assume $A_\mu$ to have a decomposition $A_\mu = \sum_{q=1}^{\Qa} \theta_q(\mu) A_q$
with non-parametric $A_q \in L^{\infty}(\Omega)$ and arbitrary $\theta_q:\Params \to \mathbb{R}$.
This gives rise to the corresponding decomposition
\begin{align}\label{eq:parameter_separable}
    a_\mu(u, v) &= \sum_{q = 1}^{\Qa} \theta_q(\mu)\, a_q(u, v),
\end{align}
of $a_\mu$, where $a_q(u,v) :=\int_{\Omega}(A_q(x) \nabla u(x))\cdot \nabla v(x)\dx$.
In case $A_\mu$ does not exhibit such a decomposition, empirical interpolation~\cite{BMNP2004}
can be used to obtain an approximate decomposition of $A_\mu$.

Finally, for $v\in V = H^1_0(\Omega)$ we introduce the following norms:
	\begin{equation*}
	\snorm{v}       := \int_\Omega \abs{\nabla v(x)}\dx, \qquad
	\anorm{v}       := \int_\Omega \abs{A_\mu^{1/2}\nabla v(x)}\dx.
	\end{equation*}
Note that $\snorm{}$ is a norm on $V$ due to Friedrich's inequality.

We remark that our approach can be easily generalized to parametric $f$ and other
boundary conditions.

\section{Localized Orthogonal Decomposition}
\label{sec:multiscale}
In this section we review the basic concepts of the LOD. 
The general idea of the LOD is to decompose the solution space into a subspace $\Vfh$ of negligible
fine-scale variations and an $a_\mu$-orthogonal low-dimensional coarse space of multiscale functions,
in which the solution is approximated.
This multiscale space is constructed by computing suitable fine-scale corrections $\QQm(u_H)$ of functions from a 
given coarse finite-element space $V_H$.
Due to the dampening of high-frequency oscillations by $a_\mu$, these corrections can then be approximated by the
solution of decoupled localized corrector problems.

In recent years, there have been various formulations of the LOD in terms of localization, interpolation,  approximation schemes and applications.
Since we are concerned with the case where storage restrictions may prevent to store explicit solutions of the corrector problems, we focus on the Petrov--Galerkin version of the LOD (PG--LOD)~\cite{elf}, which is favorable in this respect.
For further background, we refer to~\cite{LODbook}.

\subsection{Discretization and patches}
Let $V_h \subset V$ be a conforming finite-element space of dimension $\dimVh$, and let $\grid$ be the corresponding
shape regular mesh over the computational domain $\Omega$. 
We assume that the mesh size $h$ is chosen such that all features of $A_\mu$ are resolved, making a global
solution within $V_h$ infeasible. 
Further, we assume to be given a coarse mesh $\Gridh$ with mesh size $H \gg h$ that is aligned with $\grid$, and let
\begin{equation*}
    V_H := V_h \cap \mathcal{P}_1(\Gridh),     
\end{equation*}
where $\mathcal{P}_1(\Gridh)$ denotes $\Gridh$-piecewise affine functions that are continuous on $\Omega$.
We denote the dimension of $V_H$ by $\dimVH$.

For an arbitrary set $\omega \subseteq \Omega$, we define coarse grid element patches $U_k(\omega) \subset \Omega$ of size $0 \leq k \in \mathbb{N}$ by
\begin{equation*}
U_0(\omega) := \omega,  \qquad \text{and}\qquad
U_{k+1}(\omega) := \operatorname{Int}\biggl(\,\overline{\bigcup \set[T \in \Gridh]{ \overline{U_k(\omega)} \cap \overline{T} \neq \emptyset}}\,\biggr),
\end{equation*}
where $\operatorname{Int}(X)$ is the interior of the set $X$.
For a given patch size $k$, let
\begin{equation*}
    \CO := \max_{x\in\Omega} \,\#\{T \in \Gridh \,\vert\, x \in U_k(T)\}
\end{equation*}
be the maximum number of element patches overlapping in a single point of $\Omega$.

\subsection{Localized multiscale space}
In order to define the fine-scale space $\Vfh$, we consider a (quasi-)interpolation operator
$\IH : V_h \to V_H$, and let $\Vfh := \ker(\IH)$.
Multiple choices for such an interpolation operator are possible (see~\cite{MH16, PS} for an overview). 
In recent literature, using an operator that is based on local $L^2$-projections~\cite{Pe15} has proven advantageous.
However, for the following error analysis we only require that $\IH$ is linear, $\snorm{}$-continuous and idempotent
on $V_H$, i.e.,
    \begin{align*}
    \IH(v_H) &= v_H & \forall v_H &\in V_H,\\
    \snorm{\IH(v_h)} &\leq \CI \snorm{v_h} & \forall v_h &\in V_h.
    \end{align*}

Next, we define the fine-scale corrections $\QQm(v) \in \Vfh$, for a given $v_h \in V_h$, to be the solution of 
\[
a_{\mu}(\QQm(v_h),\vf) = a_{\mu}(v_h,\vf) \qquad\qquad \forall \; \vf \in \Vfh.
\]
Thus, $\QQm(v_h)$ is the $a_\mu$-orthogonal projection of $v_h$ onto $\Vfh$, and
the multiscale space
\begin{equation*}
    V_{H, \mu}^{\text{ms}} := (I - \QQm) (V_H)
\end{equation*}
is the $a_\mu$-orthogonal complement of $\Vfh$ in $V_h$, i.e.,
$V_h = V_{H,\mu}^{\text{ms}} \oplus_{a_{\mu}} \Vfh$.

However, even when $v_h$ has a small local support, $\QQm(v_h)$ will have global support,
and its computation will require the same effort as a global solution of \eqref{eq:problem_weak} in $V_h$.
Thus, we approximate $\QQm(v_h)$ using localized correctors $\QQktm(v_h) \in \Vfhkt$ in 
the patch-restricted fine-scale spaces $\Vfhkt: = \Vfh \cap H^1_0(U_k(T))$ given by
\begin{align} \label{eq:loc_cor_problems}
a_{\mu}(\QQktm(v_h),\vf) = a^T_{\mu}(v_h, \vf) \qquad \qquad \forall \; \vf \in \Vfhkt,
\end{align}
where $a^T_{\mu}$ denotes the bilinear form obtained by restricting the integration domain in the definition of $a_{\mu}$ to $T \in \Gridh$. 
We then define the localized corrector by 
\begin{equation}\label{eq:loc_corrector}
    \QQkm := \sumT\QQktm,
\end{equation}
and the localized multiscale space $V_{H,k,\mu}^{\text{ms}}$ by
\[
V_{H,k,\mu}^{\text{ms}}:= (I - \QQkm)(V_H).
\]
Due to the exponential decay of the correctors~\cite{MP14}, we choose a patch localization parameter of
$k \approx |\log H|$ to obtain a sufficient approximation.

\subsection{Petrov--Galerkin projection}
After computing $V_{H,k,\mu}^{\text{ms}}$, we determine an approximation of $u_\mu$ in this
$N_H$-dimensional multiscale space via Petrov-Galerkin projection. I.e., 
we let $u_{H,k,\mu}^{\text{ms}} \in V_{H,k,\mu}^{\text{ms}}$ be the solution of 
\begin{equation}
\label{eq:PG_}
a_{\mu}(\uhkmsm,v_H) =  F(v_H) \qquad \forall \; v_H \in V_H.
\end{equation}
To ensure that~\eqref{eq:PG_} has a unique solution, inf-sup stability of $a_\mu$ w.r.t.\ $V_{H,k,\mu}^{\text{ms}}$ and $V_H$
is required, which has been shown in~\cite{elf, HM19}.
Compared to these references, we use a slightly different definition of the inf-sup stability constant for~\eqref{eq:PG_}
by using $\snorm{}$ instead of $\anorm{}$ for the test space:
\begin{equation*}
	\CPG := \inf_{0\neq w_H\in V_H}\sup_{0\neq v_H\in V_H}
	\frac{a_\mu(w_H-\QQktm(w_H), v_H)}{\anorm{w_H-\QQktm(w_H)}\snorm{v_H}}.
\end{equation*}
For $k$ large enough, we have that
    \begin{equation}\label{eq:gamma_k}
        \CPG \approx \alpha^{1/2}\CI^{-1},
    \end{equation}
which can be proven with a simple modification of the argument in~\cite[Section 4]{HM19}.
In particular, w.l.o.g.\ we assume that $\CPG \leq \alpha^{1/2}$.

Writing the solution of~\eqref{eq:PG_} as $\uhkmsm=\uhkm - \QQkm(\uhkm)$ with $\uhkm \in V_H$,
we have the following a priori estimate, which was first shown in~\cite{elf}.
\begin{theorem}[A priori convergence result for the PG-LOD]
	\label{thm:pglod_convergence}
    For a fixed parameter $\mu \in \Params$, let 
    $u_{h,\mu} \in V_h$ be the finite-element solution of \eqref{eq:problem_weak} given by
    \begin{equation*}
        a_\mu(u_{h,\mu}, v_h) = F(v_h) \qquad\forall v_h \in V_h.
    \end{equation*}
    Then it holds that
	\begin{equation*}
        \norm{u_{h,\mu} - u_{H,k,\mu}}_{L^2} + \norm{u_{h,\mu} - u_{H,k,\mu}^{\text{ms}}}_{1} \lesssim (H + \theta^k
        k^{d/2}) \norm{f}_{L^2(\Omega)},
	\end{equation*}
	with $0 < \theta < 1 $ independent of $H$ and $k$, but dependent on the contrast $\kappa$.
\end{theorem}
Although our setting is slightly different to the one in \cite{elf} (in terms of localization and interpolation), the proof can still be followed analogously.
For a detailed discussion on the decay variable $\theta$, we refer to~\cite{MH16, HM19, MP14}.

\subsection{Computational aspects}
\label{sec:sub_comp_comp}
In order to solve \eqref{eq:PG_}, we need to solve multiple corrector problems \eqref{eq:loc_cor_problems} for every $T \in \Gridh$. 
These correctors are then used to assemble a localized multiscale matrix $\Kmu$ given by
\begin{equation} \label{eq:K}
\Kmu:= \sumT \Ktmu, \quad \left(\Ktmu \right)_{ji} := \skal{A_{\mu} (\chi_T \nabla - \nabla \QQktm) \phi_i}{\nabla \phi_j}_{U_k(T)},  
\end{equation}
where $\chi_T$ denotes the indicator function on $T$ and $\phi_i$ the finite-element basis functions of $V_H$. 
Solving \eqref{eq:PG_} then is equivalent to solving the linear system
\begin{equation}\label{eq:PG_in_matrices}
   \Kmu \cdot \uhkmsmcoeff = \mathbb{F},
\end{equation}
where $\mathbb{F}_i := F(\phi_i)$, and $\uhkmsmcoeff \in \mathbb{R}^{N_H}$ is the vector of coefficients of $\uhkmsm$
w.r.t.\ the basis of $V_{H,k,\mu}^{\text{ms}}$ corresponding to the finite-element basis $\phi_i$.

Compared to Galerkin projection onto $V_{H,k,\mu}^{\text{ms}}$, the system matrix $\Kmu$ of the Petrov-Galerkin
formulation has a smaller sparsity pattern, and we need less computational work to assemble the matrix.
Every localized corrector and thus each local contribution to $\Kmu$ can be computed in parallel without any communication and can be deleted after the contribution $\Ktmu$ has been computed.
In particular, note that the local contribution matrices $\Ktmu$ only have non-zeros in columns $i$ for which $T \subseteq \supp \phi_i$.

Overall, to compute the LOD solution the following steps are required:
\begin{enumerate}
    \item For every $T \in \Gridh$: Compute $\QQktm(\phi_i)$ by solving \eqref{eq:loc_cor_problems} for each $i$ s.t. $T \subseteq \supp \phi_i$. Assemble $\Ktmu$ according to \eqref{eq:K}. 
    \item Assemble the localized multiscale stiffness matrix $\Kmu = \sumT \Ktmu$ from the local contributions computed in step 1.
	\item Solve equation \eqref{eq:PG_in_matrices} to compute $\uhkmsmcoeff$.
\end{enumerate}
In general, neither of the steps 1--3 is computational negligible, and each of these steps has to be repeated to
obtain a solution for a new $\mu \in \Params$.
Note that step 1 requires computations on the fine-scale level whereas steps 2 and 3 solely depend on the size of the coarse mesh $\Gridh$.
In~\cite{RBLOD}, RB approximations of the local corrector problems were introduced to obtain a model that is
independent from the size of $\grid$.
However, for large $\Gridh$, the costs of solving the reduced corrector problems in step 1 and the further computations
in steps 2 and 3 still can be significant.
The two-scale reduction approach introduced in this work takes all computational steps into account and yields a
reduced order model that is independent from the sizes of both $\grid$ and $\Gridh$.

\section{Two-scale formulation of the PG--LOD}\label{sec:two_scale}
We formulate the PG--LOD method in a two-scale formulation.
In particular, we aim at considering the PG--LOD solution as the solution of one single system where the coarse system \eqref{eq:PG_} and all fine scale corrections~\eqref{eq:loc_cor_problems} are solved at the same time.
This formulation will be the basis for the Stage 2 ROM constructed in \cref{sec:stage_2_red}.

\subsection{The two-scale bilinear form}
Let $\Vts$ denote the two-scale function space given by the direct sum of Hilbert spaces
     \begin{equation*}
        \Vts:= V_H \oplus \Vfhkt[T_1] \oplus \cdots \oplus \Vfhkt[\Tlast].
    \end{equation*}
In particular, for $\uts = (u_H, \uft[T_1], \dots, \uft[\Tlast]) \in \Vts$ we define the
two-scale $H^1$-norm of $\uts$ by
    \begin{equation*}
        \Snorm{\uts}^2 := \snorm{u_H}^2 + \sumT \snorm*{\uft}^2.
    \end{equation*}
On this space, we define the two-scale bilinear form $\Btsm \in \text{Bil}(\Vts)$ given by
    \begin{multline*}
    \Btsm\left((u_H, \ufti{1}, \dots, \ufti{{\abs{\mathcal{T}_H}}}),  (v_H, \vfti{1}, \dots, \vfti{{\abs{\mathcal{T}_H}}})\right):=
    \\ a_{\mu}(u_H - \sumT \uft, v_H) + \rho^{1/2}\sumT a_{\mu}(\uft, \vft) - a_{\mu}^T(u_H, \vft),
    \end{multline*}
with a stabilization parameter $\rho \geq 1$ that will be chosen later.
Further let $\Fts \in \Vts^\prime$ be given by
    \begin{align*}
    \Fts\left((v_H, \vfti{1}, \dots, \vfti{{\abs{\mathcal{T}_H}}})\right) &:= F(v_H),
    \end{align*}
and define the two-scale solution $\utsm \in \Vts$ of of the PG--LOD by the variational problem
\begin{equation}
\Btsm\left(\utsm,  \vts \right) = \Fts(\vts) \qquad \forall \vts \in \mathcal{V}. \label{eq:two_scale_PGLOD}
\end{equation}
We show, that \eqref{eq:two_scale_PGLOD} is equivalent to the original formulation
\eqref{eq:loc_cor_problems},~\eqref{eq:PG_}:
\begin{proposition}
    The two-scale solution $\utsm \in \Vts$ of \cref{eq:two_scale_PGLOD} is uniquely determined and given by
    \begin{equation}\label{eq:two_scale_solution}
        \utsm = \left[\uhkm,\, \QQktm[T_1](\uhkm),\, \ldots,\, \QQktm[\Tlast](\uhkm)\right].
    \end{equation}
\end{proposition}
\begin{proof}
    With $\utsm$ as in~\eqref{eq:two_scale_solution} we have for any
    $\vts = (v_H, \vft[T_1], \ldots, \vft[\Tlast])$
    \begin{align*}
        \Btsm (\utsm, \vts) &=
            a_\mu(\uhkm - \sumT \QQktm(\uhkm), v_H)\\
            &\hS{50}+ \rho^{1/2} \sumT a_\mu(\QQktm(\uhkm), \vft) - a_\mu^T(\uhkm, \vft) \\
            &= F(v_H) = \Fts(\vts),
    \end{align*}
    where we have used the definition of $\QQktm(\uhkm)$ to eliminate the sum over $\Gridh$,
    the definition of $\QQkm$~\cref{eq:loc_corrector} and the definition of $\uhkm$~\cref{eq:PG_}.

    To show that $\utsm$ is the only solution of~\eqref{eq:two_scale_solution}, 
    it suffices to show that $\Btsm(\uts, \vts) = 0$ for all 
    $\vts\in \Vts$ implies $\uts = (u_H, \uft[T_1], \ldots, \uft[\Tlast]) = 0$. 
    For that, first note that for $1 \leq i \leq \Tlast$ and each $\vft[T_i] \in \Vfhkt$ we have:
    \begin{equation*}
        a_\mu(\uft[T_i], \vft[T_i]) - a_\mu^T(u_H, \vft[T_i]) =
        \rho^{-1/2}\cdot \Btsm(\uts, (0, \ldots, 0, \vft[T_i], 0 \ldots, 0)) = 0,
    \end{equation*}
    hence, $\uft[T_i] = \QQktm(u_H)$. This implies
    $a_\mu(u_H - \QQkm(u_H), v_H) = \Btsm(\uts, (v_H, 0, \ldots, 0)) = 0$
    for all $v_H \in V_H$, which means $u_H = 0$ due to the inf-sup stability of the PG--LOD bilinear
    form. However, $u_H = 0$ implies $\QQktm(u_H) = 0$ for all $T \in \Gridh$.
\end{proof}

\subsection{Analysis of the two-scale bilinear form}\label{sec:two_scale_form}

We introduce two weighted norms $\Anorm{}$, $\SMnorm{}$ on $\Vts$, w.r.t.\ which we will show the inf-sup
stability of $\Btsm$ and derive approximation error bounds.
For arbitrary $\uts = (u_H, \uft[T_1], \dots, \uft[\Tlast]) \in \Vts$
these norms are given by
\begin{align*}
    \Anorm{\uts}^2 &:= 
    \anorm{u_H - \sumT \uft}^2
    + \rho\sumT \anorm{\QQktm(u_H) - \uft}^2,\\
    \SMnorm{\uts}^2 &:= 
    \snorm{u_H}^2
    + \rho\sumT \snorm{\QQktm(u_H) - \uft}^2.
\end{align*}
    
\begin{proposition}
    $\Anorm{}$ and $\SMnorm{}$ are norms on $\Vts$ for all $\mu \in \mathcal{P}$.
\end{proposition}
\begin{proof}

    Since $\uts \mapsto u_H - \sumT \uft$ and 
    $\uts \mapsto \QQktm(u_H) - \uft$ are linear, the pull-back
    norms $\anorm{u_H - \sumT \uft}$ and $\anorm{\QQktm(u_H) - \uft}$
    are semi-norms on $\Vts$. Hence, $\Anorm{}$ is a semi-norm on $\Vts$
    as well.

    Further, we have $\snorm{u_H} = \snorm{\IH(u_H - \sumT \hS{-2} \uft)}
    \leq \CI \snorm{u_H - \sumT \hS{-2} \uft}
    \leq \CI \alpha^{-1/2}\anorm{u_H - \sumT \hS{-2} \uft}$, so
    $\Anorm{u} = 0$ implies $u_H = 0$. This in turn implies
    $\anorm{\uft} = \anorm{\QQktm(u_H) - \uft} = 0$ for all $T \in \Gridh$,
    so $\utsm = 0$. Hence, $\Anorm{}$ is indeed a norm on $\Vts$.
    The argument for $\SMnorm{}$ is similar.
\end{proof}

We are going to show that $\Anorm{}$ and $\SMnorm{}$ are equivalent norms, for which we require some technical results.

\begin{lemma} \label{lemma:norm_prop_1}
	Let $\vft\in\Vfhkt$ for each $T \in \Gridh$ be given. Then we have:
	\begin{equation*}
	\snorm{\sumT\vft}^2 \leq \CO \sumT\snorm{\vft}^2 \quad\text{and}\quad
	\anorm{\sumT\vft}^2 \leq \CO \sumT\anorm{\vft}^2.
	\end{equation*}
\end{lemma}
\begin{proof}
	Using Jensen's inequality, we have:
	\begin{equation*}
	\begin{aligned}
	\anorm{\sumT\vft}^2 &= \int_\Omega \abs{\sumT A_\mu^{1/2}(x)\nabla\vft(x)}^2\dx \\
	&\leq \int_\Omega \CO\cdot \sumT\abs{A_\mu^{1/2}(x)\nabla\vft(x)}^2\dx
	= \CO \sumT\anorm{\vft}^2.
	\end{aligned}
	\end{equation*}
	The proof for $\snorm{}$ is the same.
\end{proof}

\begin{lemma}\label{thm:qqktm_bound}
    For arbitrary $u_H \in V_H$ we have
     \begin{equation*}
         \left(\sumT \anorm{\QQktm (u_H)}^2\right)^{1/2} \leq \anorm{u_H}.
     \end{equation*}
\end{lemma}
\begin{proof}
By definition of $\QQktm$ we have
    \begin{equation*}
    \begin{split}
    \sumT &\anorm{\QQktm (u_H)}^2 
    = \sumT a_\mu(\QQktm (u_H), \QQktm (u_H)) 
    = \sumT a_\mu^T(u_H, \QQktm (u_H)) \\
    &\!\!\!\!\!\!\!\leq \sumT \biggl(\int_T\abs{A_\mu^{1/2}(x)\nabla u_H(x)}^2\dx\biggr)^{1/2}
    \cdot \biggl(\int_T\abs{A_\mu^{1/2}(x)\nabla\QQktm(u_H)(x)}^2\dx\biggr)^{1/2} \\
    &\!\!\!\!\!\!\!\leq \biggl(\sumT \int_T\abs{A_\mu^{1/2}(x)\nabla u_H(x)}^2\dx\biggr)^{1/2}
    \cdot \biggl(\sumT \int_T\abs{A_\mu^{1/2}(x)\nabla\QQktm(u_H)(x)}^2\dx\biggr)^{1/2} \\
    &\!\!\!\!\!\!\!\leq \anorm{u_H}^2 \cdot \biggl( \sumT \anorm{\QQktm u_H}^2 \biggr)^{1/2}.
    \end{split}
    \end{equation*}
    Dividing by the second factor yields the claim.
\end{proof}

\noindent Now, we are prepared to show the equivalence of both norms.
\begin{proposition}\label{thm:norm_equiv}
    $\Anorm{}$ and $\SMnorm{}$ are equivalent norms on $\Vts$
    with the following bounds for every $\uts \in \Vts$:
    \begin{equation*}
        \CI^{-1}\alpha^{1/2}\SMnorm{\uts}
        \leq \Anorm{\uts}
        \leq \sqrt{3}(1+\CO)^{1/2}\beta^{1/2}\SMnorm{\uts}.
        \label{eq:SMnormAnorm}
    \end{equation*}
\end{proposition}
\begin{proof}
    Let $\uts = (u_H, \uft[T_1], \ldots, \uft[\Tlast])$. To bound $\SMnorm{\uts}$
    by $\Anorm{\uts}$, note that
    \begin{equation}\label{eq:two_scale_norm_equiv_proof_sa_bound}
        \snorm{u_H}^2 \leq \CI^2 \snorm{u_H - \sumT\uft}^2
        \leq \alpha^{-1}\CI^2\anorm{u_H - \sumT\uft}^2,
    \end{equation}
    and, using $\CI \geq 1$,
    \begin{equation*}
        \rho\sumT\snorm{\QQktm(u_H) - \uft}^2 \leq
        \alpha^{-1}\CI^2\rho\sumT\anorm{\QQktm(u_H) - \uft}^2.
    \end{equation*}
    To bound $\Anorm{\uts}$ by $\SMnorm{\uts}$ we use \cref{lemma:norm_prop_1}, \cref{thm:qqktm_bound} and
    $\rho \geq 1$ to obtain
    \begin{align*}
        &\anorm{u_H - \sumT\uft}^2\\
        &\qquad\leq 3\anorm{u_H}^2 + 3\CO\sumT\anorm{\QQktm(u_H) - \uft}^2 + 3\CO\sumT\anorm{\QQktm(u_H)}^2 \\
        &\qquad\leq 3(1 + \CO)\anorm{u_H}^2 + 3\CO\rho\sumT\anorm{\QQktm(u_H) - \uft}^2\\
        &\qquad\leq 3(1 + \CO)\beta\snorm{u_H}^2 + 3\CO\beta\rho\sumT\snorm{\QQktm(u_H) - \uft}^2.
    \end{align*}
    Adding
    \begin{equation*}
        \rho\sumT\anorm{\QQktm(u_H) - \uft}^2 \leq
        \beta\rho\sumT\snorm{\QQktm(u_H) - \uft}^2
    \end{equation*}
    to both sides yields the claim.
\end{proof}

Finally, we show that $\Btsm$ is $\Anorm{}$-$\Snorm{}$ inf-sup stable with controllable constants.
\begin{proposition}\label{thm:two_scale_inf_sup}
    Let
    \begin{equation*}
        \rho := \CO\cdot\kappa,
    \end{equation*}
    then $\Btsm$ is $\Anorm{}$-$\Snorm{}$-continuous and inf-sup stable
    with the following bounds on the respective constants:
    \begin{equation*}
        \sup_{0 \neq \uts \in \Vts}
        \sup_{0 \neq \vts \in \Vts}
        \frac{\Btsm(\uts,\vts)}{\Anorm{u}\cdot\Snorm{v}} 
        \leq \beta^{1/2} \quad \text{and}\quad
        \inf_{0 \neq \uts \in \Vts}
        \sup_{0 \neq \vts \in \Vts}
        \frac{\Btsm(\uts,\vts)}{\Anorm{u}\cdot\Snorm{v}} 
        \geq \CPG/\sqrt{5}.
    \end{equation*}
\end{proposition}
\begin{proof}
    We first bound the continuity constant of $\Btsm$. To this end, let
    $\uts = (u_H, \uft[T_1], \ldots, \uft[\Tlast]) \in \Vts$ and 
    $\vts = (v_H, \vft[T_1], \ldots, \vft[\Tlast]) \in \Vts$ be arbitrary.
    Then we have
    \begin{equation*}
        \begin{aligned}
            \Btsm(\uts, \vts)
            & = a_\mu(u_H - \sumT\uft, v_H)
            - \rho^{1/2}\sumT\left(a_\mu(\uft,\vft) - a_\mu^T(u_H,\vft)\right)\\
            & \leq \anorm{u_H-\sumT\uft}\anorm{v_H}
            + \sumT\rho^{1/2}\anorm{\uft-\QQktm(u_H)}\anorm{\vft}\\
            & \leq \left[\anorm{u_H - \sumT\uft}^2
            + \sumT\rho\anorm{\uft-\QQktm(u_H)}^2 \right]^{1/2}\\
            & \qquad\qquad\cdot\left[\anorm{v_H}^2 + \sumT\anorm{\vft}^2\right]^{1/2}\\
            & \leq \Anorm{\uts} \cdot \beta^{1/2} \cdot \Snorm{\vts}.
        \end{aligned}
    \end{equation*}
    To prove inf-sup stability, first note that
        \begin{align*}
            \anorm{u_H-\sumT\uft}\hspace{-6em}& \\
            & \leq \anorm{u_H-\sumT\QQktm(u_H)}
            + \anorm{\sumT(\QQktm(u_H) - \uft)} \\
            & \leq \CPG^{-1} \sup_{0\neq v_H\in V_H} \frac{a_\mu(u_H-\sumT\QQktm(u_H),v_H)}{\snorm{v_H}}
            + \anorm{\sumT(\QQktm(u_H) - \uft)} \\
            & \leq \CPG^{-1}\underbrace{\sup_{0\neq v_H\in V_H}\hspace{-2pt}
            \frac{\Btsm(\uts,(v_H,0,\ldots,0))}{\Snorm{(v_H,0,\ldots,0)}}}_{A}
            \\&\quad+ \CPG^{-1}\hspace{-2pt}\sup_{0\neq v_h\in V_H}\hspace{-2pt}
            \frac{a_\mu(\sumT(\QQktm(u_H)-\uft),v_H)}{\snorm{v_H}} + \anorm{\sumT\hspace{-2pt}(\QQktm(u_H) - \uft)} \\
            & \leq \CPG^{-1}A
            + \left(\CPG^{-1}
            \sup_{0\neq v_H\in V_H}\frac{\anorm{v_H}}{\snorm{v_H}}
            + 1\right) \anorm{\sumT(\QQktm(u_H) - \uft)} \\
            & \leq \CPG^{-1}A 
            + (\CPG^{-1}\beta^{1/2} + 1)\CO^{1/2}\left(\sumT\anorm{\QQktm(u_H) - \uft}^2\right)^{1/2},
        \end{align*}
    hence:
    \begin{equation*}
        \begin{aligned}
            \Anorm{\uts}^2
            & = \anorm{u_H-\sumT\uft}^2 + \rho\sumT\anorm{\QQktm(u_H)-\uft}^2\\
            & \leq 2\CPG^{-2}A^2
            + (2\CPG^{-2}\beta\CO +2\CO + \rho)\sumT\anorm{\QQktm(u_H) - \uft}^2.
        \end{aligned}
    \end{equation*} 
    Further,
    \begin{equation*}
        \begin{aligned}
            \left(\sumT\anorm{\uft - \QQktm(u_H)}^2\right)^{1/2}
            &= \frac{\sumT a_\mu(\uft-\QQktm(u_H), \uft-\QQktm(u_H))}{(\sumT\anorm{\uft-\QQktm(u_H)}^2)^{1/2}}\\
            & \leq \sup_{\vft\in\Vfhkt}\frac{\sumT a_\mu(\uft-\QQktm(u_H), \vft)}{(\sumT\anorm{\vft}^2)^{1/2}}\\
            & \leq \alpha^{-1/2}\sup_{\vft\in\Vfhkt}\frac{\sumT a_\mu(\uft, \vft) - a_\mu^T(u_H, \vft)}{(\sumT\snorm{\vft}^2)^{1/2}}\\
            & = \alpha^{-1/2}\rho^{-1/2}\underbrace{\sup_{\vft\in\Vfhkt}
            \frac{\Btsm(\uts, (0,\vft[T1],\ldots,\vft[\Tlast]))}
            {\Snorm{(0,\vft[T1],\ldots,\vft[\Tlast])}}}_{B}.
        \end{aligned}
    \end{equation*}
    Combining both estimates yields
    \begin{equation*}
        \begin{aligned}
            \Anorm{\uts}^2
            &\leq 2\CPG^{-2}A^2 + (2\CPG^{-2}\alpha^{-1}\beta\CO\rho^{-1} +
            2\alpha^{-1}\CO\rho^{-1}+\alpha^{-1})B^2\\
            &= 2\CPG^{-2}A^2 + \underbrace{(2\CPG^{-2} +
            2\alpha^{-1}\kappa^{-1}+\alpha^{-1})}_{\leq 5\cdot\CPG^{-2}}B^2.\\
            &\leq 5\cdot\CPG^{-2}(A^2 + B^2) \\
            &=5\cdot\CPG^{-2}\left(\sup_{\vts\in\Vts}\frac{\Btsm(\uts,\vts)}{\Snorm{\vts}}\right)^2,
        \end{aligned}
    \end{equation*}
    where we have used $\CPG \leq \alpha^{1/2}$ and $\kappa \geq 1$ in the third equality.
    In the last equality we have used the fact that the square norm of a linear functional on a direct sum
    of Hilbert spaces is the sum of the square norms of the functional restricted to the respective subspaces.
\end{proof}

\subsection{Error Bounds}
\label{sec:a-posteriori_estimators}
Exploiting the inf-sup stability of $\Btsm$, we now easily obtain error bounds w.r.t.\ the PG-LOD solution.
We start with an a posteriori bound and define for 
arbitrary $\uts \in \Vts$ the residual-based error indicators
\begin{align}\label{eq:stage2_est_a}
        \eta_{a,\mu}(\uts) &:= 
        \sqrt{5}\CPG^{-1}
        \sup_{v\in\Vts}
        \frac{\Fts(\vts) - \Btsm(\uts,\vts)}
        {\Snorm{\vts}},\\\label{eq:stage2_est_s}
        \eta_{1,\mu}(\uts) &:= 
        \sqrt{5}\CI\alpha^{-1/2}\CPG^{-1}
        \sup_{v\in\Vts}
        \frac{\Fts(\vts) - \Btsm(\uts,\vts)}
        {\Snorm{\vts}}.
\end{align}
These error indicators provide strict and efficient upper bounds on the error between $\uts$ and the 
two-scale PG-LOD solution:
\begin{theorem}[A posteriori bound]\label{thm:a_posteriori}
    Let $\uts= (u_H, \uft[T_1], \ldots, \uft[\Tlast])\in\Vts$ be an arbitrary two-scale function, denote by $\utsm$ the
    solution of the two-scale solution as in \cref{eq:two_scale_solution} for a
    given parameter $\mu$, and let $\rho$ be given as in \cref{thm:two_scale_inf_sup}.
    Then the following energy error bounds hold:
    \begin{equation}\label{eq:a_posteriori_anorm}
        \Anorm{\utsm - \uts}
        \leq \eta_{a,\mu}(\uts)
        \leq \sqrt{5}\CPG^{-1}\beta^{1/2} \Anorm{\utsm - \uts}.
    \end{equation}
    Further, we have 
    \begin{equation}\label{eq:a_posteriori_snorm}
        \left(\snorm{\uhkm - u_H}^2 + \rho\sumT\snorm{\QQktm(u_H) - \uft}^2\right)^{1/2}
        \leq \eta_{1,\mu}(\uts),
    \end{equation}
    and
    \begin{equation}\label{eq:a_posteriori_snorm_eff}
    \begin{split}
            \eta_{1,\mu}(\uts) &\leq \sqrt{15}\CI(\CO+1)^{1/2}\kappa^{1/2}\CPG^{-1}\beta^{1/2}\\&\qquad\qquad\cdot
    \left(\snorm{\uhkm - u_H}^2 + \rho\sumT\snorm{\QQktm(u_H) - \uft}^2\right)^{1/2}.
    \end{split}
    \end{equation}
\end{theorem}
\begin{proof}
    Since $\utsm$ is a solution of \cref{eq:two_scale_PGLOD} we have
    \begin{equation*}
        \Btsm(\utsm - \uts,\vts) =
        \Fts(\vts) - \Btsm(\uts, \vts).
    \end{equation*}
    Hence, \cref{eq:a_posteriori_anorm} directly follows from \cref{thm:two_scale_inf_sup}.
    \Cref{eq:a_posteriori_snorm,eq:a_posteriori_snorm_eff} follow from \cref{eq:a_posteriori_anorm}
    using \cref{thm:norm_equiv} and noting that for each $T \in \Gridh$ we have 
    \begin{equation*}
        \snorm{\QQktm(\uhkm-u_H) - (\QQktm(\uhkm)-\uft)}=\snorm{\QQktm(u_H)-\uft}.
    \end{equation*}~
\end{proof}
Finally, we also show a corresponding a priori result:
\begin{theorem}[A priori bound]\label{thm:a_priori}
    Let $\overline{\Vts}$ be an arbitrary linear subspace of $\Vts$ and
    let $\overline{\uts}$ be the solution of the residual-minimization problem
    \begin{equation}\label{eq:red_solution_abstract}
        \overline{\uts}_\mu:= \argmin_{\uts\in\overline{\Vts}}
        \sup_{\vts\in \Vts}\frac{\Fts(v) - \Btsm(\uts, \vts)}
        {\Snorm{\vts}},
    \end{equation}
    then we have
    \begin{equation*}
        \Anorm{\utsm - \overline{\uts}_\mu}
        \leq \sqrt{5}\CPG^{-1}\beta^{1/2}
        \min_{\overline{\vts}\in\overline{\Vts}}\Anorm{\utsm - \overline{\vts}},
    \end{equation*}
    and
    \begin{equation*}
        \begin{split}
            \left(\snorm{\uhkm - u_H}^2 + \rho\sumT\snorm{\QQktm(u_H) - \uft}^2\right)^{1/2}\hspace{-20em}&\\
            &\leq \sqrt{15}\CI(\CO+1)^{1/2}\kappa^{1/2}\CPG^{-1}\beta^{1/2} \\&\qquad\qquad\cdot
            \min_{\overline{\vts}\in\overline{\Vts}}
            \left(\snorm{\uhkm - \overline{v}_H}^2 + \rho\sumT\snorm{\QQktm(u_H) - \overline{v}_T^\text{f}}^2\right)^{1/2}.
        \end{split}
    \end{equation*}
\end{theorem}
\begin{proof}
    This follows directly from \cref{thm:a_posteriori} and the definition
    of $\overline{\uts}_\mu$.
\end{proof}

Note that for $k$ large enough, we have $\CPG \approx \alpha^{1/2}\CI^{-1}$
such that the a priori and a posteriori bounds have
efficiencies of that scale with $\kappa^{1/2}$ in the energy norm and with
$\kappa$ in the 1-norm, which agrees with what is to be expected for
these error bounds in a standard finite element setting.

\section{Reduced Basis approach}
\label{sec:MOR} 
In this section we describe the construction of the TSRBLOD ROM in detail.
As with all RB methods, the ROM is defined by a projection of the original model
equations, in our case the two-scale formulation~\cref{eq:two_scale_PGLOD}, onto
a reduced approximation space (\cref{sec:stage_2_rom}).
Rigorous upper and lower bounds for the MOR error are given by a residual-based
a posteriori error estimator (\cref{sec:stage_2_err}).
In order to be able to quickly assemble the ROM for a new parameter $\mu$ and to
efficiently evaluate the error estimator, an offline-online decomposition of the
ROM must be performed (\cref{sec:stage_2_offon}).
As part of the offline phase, the reduced space is constructed
as the linear span of FOM solutions $\mathfrak{u}_{\mu^*}$, where the snapshot parameters $\mu^*$ are selected
via an iterative greedy search over $\Params$ (\cref{sec:stage_2_bas_gen}).

The computation of the solution snapshots $\mathfrak{u}_{\mu^*}$ via \cref{eq:two_scale_solution}, however,
requires for each new $\mu^*$ the recomputation of all corrector problems. For large
problems this may be computationally infeasible.
Thus, we combine our approach (Stage 2) with a preceeding prepartory step similar
to~\cite{RBLOD}, where each corrector problem is replaced by an efficient ROM surrogate (Stage 1).
These ROMs are then used in the offline phase of Stage 2 to compute approximate solution snapshots
$\mathfrak{u}_{\mu^*}$.
Again, Stage 1 is divided into ROM construction (\cref{sec:stage_1_rom}), error estimation
(\cref{sec:stage_1_err}), offline-online decomposition (\cref{sec:stage_1_of/on}) and
construction of the reduced spaces (\cref{sec:stage1_basisgen}).
After a Stage 1 ROM is constructed, all associated fine-mesh data can be deleted.
In particular, all computations in Stage 2 are independent from size of $\grid$.

We emphasize again that, compared to~\cite{RBLOD}, the Stage 2 TSRBLOD ROM is not only
independent from $\grid$, but also from the number of coarse elements. Its size only depends
on the number of selected basis vectors in~\cref{sec:stage_2_bas_gen}.
Further, in contrast to~\cite{RBLOD}, the derived error estimator fully takes the effect of the errors of 
the reduced corrector problems on the resulting global solution into account and, thus,
rigorously bounds the error of the ROM w.r.t.\ the LOD solution.

\subsection{Stage 1: Computing RB approximations of $\QQktm$ and $\Ktmu$}
\label{sec:stage1}
\subsubsection{Definition of the reduced order model}\label{sec:stage_1_rom}%
    Let $T \in \Gridh$ be fixed,
    and let $\Vfrbkt \subset \Vfhkt$ be an approximation space for the correctors $\QQktm(v_H)$ of $v_H$
    for arbitrary $v_H \in V_H$ and $\mu \in \mathcal{P}$.
    Then, for given $v_H \in V_H$ and $\mu \in \mathcal{P}$, we determine an approximate corrector $\QQktmrb(v_H) \in
    \Vfrbkt$ via Galerkin projection onto $\Vfrbkt$ as the solution of
    \begin{equation}\label{eq:stage1_rom}
        a_\mu(\QQktmrb(v_H), \vft) = a_\mu^T(v_H, \vft) \qquad \forall \vft \in
        \Vfrbkt.
	\end{equation}

    Note that $\QQktmrb(v_H)$ is well-defined since $a_\mu$ is a coercive bilinear form.
    Using these reduced correction operators we can define an approximate localized multiscale matrix
    $\Kmurb$ given by
    \begin{equation*}
    \Kmurb:= \sumT \Ktmurb, \quad \left(\Ktmurb \right)_{ji} := \skal{A_{\mu} (\chi_T \nabla - \nabla \QQktmrb)
    \phi_i}{\nabla \phi_j}_{U_k(T)}.
    \end{equation*}

\subsubsection{Error estimation}\label{sec:stage_1_err}
We employ standard RB tools for the a posteriori error estimation of the reduced system \eqref{eq:stage1_rom}.
In detail, we use the residual-norm based estimate:
\begin{equation}\label{eq:stage1_error_estimate}
        \anorm{\QQktm(v_H) - \QQktmrb(v_H)} 
        \leq \eta_{T,\mu}(\QQktmrb(v_H))
        \leq \kappa^{1/2}\anorm{\QQktm(v_H) - \QQktmrb(v_H)},
\end{equation}
where
\begin{equation*}
    \eta_{T,\mu}(\QQktmrb(v_H)) := \alpha^{-1/2} \sup_{\vft\in\Vfhkt}
                     \frac{a_\mu^T(v_H, \vft) - a_\mu(\QQktmrb(v_H), \vft)}{\snorm{\vft}}.
\end{equation*}
The bounds in \eqref{eq:stage1_error_estimate} easily follow from the definition
of $\QQktm(v_H)$ and the equivalence of $\anorm{}$ and $\snorm{}$.

\subsubsection{Offline-Online Decomposition}\label{sec:stage_1_of/on}
In order to compute $\QQktmrb(v_H)$ and subsequently $\Ktmurb$, let $\dimVfrbkt := \dim \Vfrbkt$, and choose a basis
$\rbt{n}$, $1 \leq n \leq \dimVfrbkt$, of $\Vfrbkt$.
Expanding $\QQktmrb(v_H)$ w.r.t.\ this basis as
\begin{equation*}
    \QQktmrb(v_H) := \sum_{n=1}^{\dimVfrbkt} c_n\cdot\rbt{n},
\end{equation*}
the coefficient vector $c \in \mathbb{R}^{\dimVfrbkt}$ is given as the solution of the $\dimVfrbkt \times
\dimVfrbkt$-dimensional linear system
\begin{equation}\label{eq:stage1_rom_in_matrices}
    \mathbb{A}^T_\mu \cdot c = \mathbb{G}^T_\mu(v_H),
\end{equation}
where
\begin{equation*}
    (\mathbb{A}^T_{\mu})_{m,n} := a_\mu(\rbt{n}, \rbt{m}) \qquad\text{and}\qquad
    \mathbb{G}^T_\mu(v_H)_m := a^T_\mu(v_H, \rbt{m}).
\end{equation*}
While~\eqref{eq:stage1_rom_in_matrices} can be solved quickly when $\dimVfrbkt$ is sufficiently small, we still need to
re-assemble this equation system for each new parameter $\mu$ and coarse-scale function $v_H$.
This requires time and memory that scales with $\dim \Vfhkt$.
To avoid these high-dimensional computations we exploit \eqref{eq:parameter_separable} and pre-assemble
matrices and vectors
\begin{equation*}
    (\mathbb{A}^T_{q})_{m,n} := a_q(\rbt{n}, \rbt{m}) \qquad\text{and}\qquad
    (\mathbb{G}^T_{q,j})_{m} := a^T_q(\phi_{i_{T,j}}, \rbt{m}),
\end{equation*}
for $1 \leq q \leq \Qa$ and $1 \leq j \leq \KT$, where $\KT$ 
is the number of finite-element basis functions $\phi_{i}$ of $V_H$ with support containing $T$ and $i_{T,1},
\ldots, i_{T,\KT}$ is an enumeration of these basis functions. 
Then, $\mathbb{A}^T_\mu$ and $\mathbb{G}^T_\mu(v_H)$ can be determined as
\begin{equation*}
    \mathbb{A}^T_{\mu} := \sum_{q=1}^{\Qa}\theta_q(\mu) \mathbb{A}^T_{q} \qquad\text{and}\qquad
    \mathbb{G}^T_\mu(v_H) := \sum_{q=1}^{\Qa}\sum_{j=1}^{\KT}\theta_q(\mu)\lambda_{i_{T,j}}(v_H) \mathbb{G}^T_{q,j},
\end{equation*}
where by $\lambda_{i}\in V_H^\prime$ we denote the dual basis of $\phi_i$.

To compute $\Kmurb$, we further store the matrices 
\begin{equation*} 
    \left(\Ktqo\right)_{j,i} := \skal{A_q \chi_T \nabla \phi_i}{\nabla \phi_j}_{U_k(T)}
    \qquad\text{and}\qquad
    \left(\Ktqrb\right)_{j,n} := \skal{A_{q} \nabla \rbt{n}}{\nabla \phi_j}_{U_k(T)}.
\end{equation*}
Then we have:
\begin{equation*} 
    \left(\Ktmurb\right)_{j,i} = 
    \sum_{q=1}^{\Qa}\theta_q(\mu)\left[\left(\Ktqo\right)_{j,i} -
    \sum_{n=0}^{N_T} c^i_n(\mu) \left(\Ktqrb\right)_{j,n}\right],
\end{equation*}
where $c^i(\mu) \in \mathbb{R}^{N_T}$ is given as the solution of
\begin{equation*}
    \mathbb{A}^T_\mu \cdot c^i(\mu) = \mathbb{G}^T_\mu(\phi_i).
\end{equation*}
Note that $(\mathbb{K}^T_{q,0})_{j,i}$, $(\mathbb{K}^T_q)_{j,n}$ and $c^i(\mu)$ are zero unless the support of $\phi_i$
is non-disjoint from $T$ and the support of $\phi_j$ is non-disjoint from $U_k(T)$.
In particular, only $\KT$ reduced problems have to be solved in order to determine $\Ktmurb$.
The total computational effort for solving these problems is of order $\mathcal{O}(\Qa N_T^2 + N_T^3 + \KT N_T^2)$, where
the first term corresponds to the assembly of $\mathbb{A}_\mu^T$, the second term to its LU decomposition and third term
to the solution of the $\KT$ linear systems using forward/backward substitution.

Finally, to efficiently evaluate $\eta_{T,\mu}(\QQktmrb(v_H))$, first note that
\begin{equation*}
    \eta_{T,\mu}(\QQktmrb(v_H)) = \alpha^{-1/2} \norm{\mathcal{R}_{\Vfhkt}(a_\mu^T(v_H, \cdot) - a_\mu(\QQktmrb(v_H),
    \cdot))}_1,
\end{equation*}
where $\mathcal{R}_{\Vfhkt}: (\Vfhkt)^\prime \to \Vfhkt$ denotes the Riesz isomorphism.

Following~\cite{BuhrEngwerEtAl2014}, let $\Wfrbkt$ denote the $M_T$-dimensional linear subspace of $\Vfhkt$ that is
spanned by the vectors
\begin{align*}
    \set[\mathcal{R}_{\Vfhkt}(a_q(\rbt{n}, \cdot))]{1\leq q \leq \Qa, 1\leq n \leq N_T} \,\cup\,\hspace{-20em}& \\&
    \qquad\qquad\set[\mathcal{R}_{\Vfhkt}(a_q^T(\phi_{i_{T,j}}, \cdot))]{1\leq q \leq \Qa, 1\leq j \leq \KT}.
\end{align*}
Choose an $H^1$-orthonormal basis $\rbestt{m}$ of $\Wfrbkt$.
Since $\mathcal{R}_{\Vfhkt}(a_\mu^T(v_H, \cdot) - a_\mu(u_N, \cdot))$ lies in $\Wfrbkt$ for each $u_N \in \Vfrbkt$, we
obtain
\begin{equation*}
    \begin{aligned}
        \eta_{T,\mu}(\QQktmrb(v_H)) 
        &= \alpha^{-1/2} \norm*{\left[\left(\mathcal{R}_{\Vfhkt}(a_\mu^T(v_H, \cdot) - a_\mu(\QQktmrb(v_H), \cdot)),
        \rbestt{m}\right)_1\right]^{M_T}_{m=1}} \\
        &=\alpha^{-1/2} \norm*{\left[a_\mu^T(v_H, \rbestt{m}) - a_\mu(\QQktmrb(v_H), \rbestt{m})\right]^{M_T}_{m=1}}.
    \end{aligned}
\end{equation*}
Hence, defining the $M_T$-dimensional vectors and $M_T \times N_T$ matrices
\begin{equation}
    (\hat{\mathbb{G}}^T_{q,j})_{m} := a^T_q(\phi_{i_{T,j}}, \rbestt{m})\qquad\text{and}\qquad
    (\hat{\mathbb{A}}^T_{q})_{m,n} := a_q(\rbt{n}, \rbestt{m}),
\end{equation}
we obtain
\begin{equation}\label{eq:stage1_error_estimate_in_matrices}
    \eta_{T,\mu}(\QQktmrb(v_H)) = \alpha^{-1/2}
    \norm*{
        \sum_{q=1}^{\Qa}\theta_q(\mu)\left(\sum_{j=1}^{\KT}\lambda_{i_{T,j}}(v_H) \hat{\mathbb{G}}^T_{q,j}
        - \hat{\mathbb{A}}^T_{q} \cdot c\right)
    },
\end{equation}
where $c$ is again given by \eqref{eq:stage1_rom_in_matrices}.
We remark that $M_T$ can be bounded by $\Qa(N_T + \KT)$.
Hence, the cost of evaluating \eqref{eq:stage1_error_estimate_in_matrices} is of order 
$\mathcal{O}(\Qa^2(N_T + \KT)^2)$.

\subsubsection{Basis generation}\label{sec:stage1_basisgen}

To build the reduced spaces $\Vfrbkt$, we use a standard weak greedy \cite{BinevCohenEtAl2011} approach in order to 
minimize the model order reduction error $\QQktm(v_H) - \QQktmrb(v_H)$ for all $\mu \in \Params$ and $v_H \in V_H$.
To this end, we choose a target error tolerance $\varepsilon_1$ and an appropriate training set of parameters
$\ParamsTrain$, over which we estimate the maximum reduction error.
The initial reduced space $\Vfrbkt$ is chosen as the zero-dimensional space.
Then, in each iteration, the reduction error is estimated with $\eta_{T,\mu}(\QQktmrb(v_H))$ for all $\mu \in
\ParamsTrain$ and $\phi_{i_{T,j}}$, and a pair $\mu^*$, $\phi_{i_{T,j^*}}$ maximising the estimate is selected.
Thanks to the offline-online decomposition of $\eta_{T,\mu}$ this step does not involve any high-dimensional
computations, so $\ParamsTrain$ can be chosen large.
Since~\eqref{eq:stage1_rom} is linear, it suffices to consider the basis functions $\phi_{i_{T,j}}$ as $v_H$.
After $\mu^*$, $j^*$ have been found, $\QQktmrb[T][\mu^*](\phi_{i_{T,j^*}})$ is computed. 
$\Vfrbkt$ is extended with this solution \emph{snapshot}, and the offline-online decompostion for this extended reduced
space is computed.
The iteration ends, when the maximum estimated error drops below $\varepsilon_1$.
A formal definition of the procedure is given in Algorithm~\ref{alg:stage1_greedy}.

\begin{algorithm2e}
   \KwData{$T$, $\ParamsTrain$, $\varepsilon_1$}
   \KwResult{$\Vfrbkt$}
   $\Vfrbkt \leftarrow \{0\}$\;
   \While{$\max_{\mu\in\ParamsTrain}\max_{1\leq j\leq \KT}\eta_{T,\mu}(\QQktmrb(\phi_{i_T,j})) > \varepsilon_1$}{
       $(\mu^*, j^*) \leftarrow \argmax_{(\mu^*, j^*) \in \ParamsTrain \times \{1,\ldots,\KT\}}\eta_{T,\mu}(\QQktmrb(\phi_{i_T,j}))$\;
       $\Vfrbkt \leftarrow \Span(\Vfrbkt \cup \{\QQktm[T][\mu^*](\phi_{i_{T,j^*}})\})$\;
   }
   \caption{Weak greedy algorithm for the generation of $\Vfrbkt$.}\label{alg:stage1_greedy}
\end{algorithm2e}

\subsection{Stage 2: Computing RB approximations of $\utsm$}
\label{sec:stage_2_red}

\subsubsection{Definition of the reduced order model}\label{sec:stage_2_rom}
In order to find an approximate solution of $\utsm$, we assume to be given an appropriate reduced subspace $\Vtsrb$ of
$\Vts$. 
As we have proven inf-sup stability of $\Btsm$ in~\cref{thm:two_scale_inf_sup}, and since the inf-sup stability is
preserved by restricting $\Btsm$ to a linear subspace, we can define the reduced two-scale solution $\utsmrb$ as the
unique solution of the residual minimization problem
\begin{equation}\label{eq:stage2_rom}
    \utsmrb := \argmin_{\uts \in \Vtsrb} \sup_{\vts \in \Vts} \frac{\Fts(\vts) - \Btsm(\uts, \vts)}{\Snorm{\vts}}.
\end{equation}
As a direct consequence of~\cref{thm:a_priori}, $\utsmrb$ is a quasi best-approximation of $\utsm$ within $\Vtsrb$. 

\subsubsection{Error estimation}\label{sec:stage_2_err}
We have already defined a posteriori error estimators for approximations of the two-scale solution $\utsm$
in~\cref{sec:a-posteriori_estimators}.
In~\cref{thm:a_posteriori} we have shown that these estimators yield efficient upper bounds for the approximation errors
in the two-scale energy norm as well as in the Sobolev 1-norm.
Note that even though we will use the Stage 1 approximations $\Ktmurb$ of $\Ktmu$ to build the reduced space $\Vtsrb$,
the derived error estimates are with respect to the true LOD solution and take these approximation errors into account.

\subsubsection{Offline-Online Decomposition}\label{sec:stage_2_offon}
For the offline-online decomposition of \eqref{eq:stage2_rom}, we proceed similar to the decomposition of the Stage~1
error estimator $\eta_{T,\mu}$.
Denote by $\mathcal{R}_{\Vts}: \Vts^\prime \to \Vts$ the Riesz isomorphism for $\Vts$.
Then~\eqref{eq:stage2_rom} is equivalent to solving
\begin{equation}\label{eq:stage2_rom_riesz}
    \utsmrb := \argmin_{\uts \in \Vtsrb} \Snorm{\mathcal{R}_{\Vts}(\Fts) - \mathcal{R}_{\Vts}(\Btsm(\uts, \cdot))}^2.
\end{equation}
Let $N := \dim \Vtsrb$, and let $\rbts{n}$, $1 \leq n \leq N$ be a basis of $\Vtsrb$. 
We again construct an $\Snorm{}$-orthonormal basis $\rbestts{m}$ for the $M$-dimensional subspace $\Wtsrb$ of $\Vts$
spanned by the vectors
\begin{equation}\label{eq:two_scale_estimator_generators}
    \{\mathcal{R}_{\Vts}(\Fts)\} \cup \{\mathcal{R}_{\Vts}(\Btsq(\rbts{n}, \cdot)) \,|\, 1\leq n \leq N, 1 \leq q
    \leq \Qa \},
\end{equation}
where
\begin{equation*}
    \begin{split}
    \Btsq\left((u_H, \ufti{1}, \dots, \ufti{{\abs{\mathcal{T}_H}}}),  (v_H, \vfti{1}, \dots, \vfti{{\abs{\mathcal{T}_H}}})\right) &:=
    \\ a_{q}(u_H - \sumT \uft, v_H) + &\rho^{1/2}\sumT a_{q}(\uft, \vft) - a_{q}^T(u_H, \vft),
    \end{split}
\end{equation*}
such that $\Btsm$ has the decomposition:
$
    \Btsm = \sum_{q=1}^{\Qa} \theta_q(\mu) \Btsq.
$
Using these bases, we define matrices $\hat{\mathbb{A}}_q \in \mathbb{R}^{M\times N}$ and the vector
$\hat{\mathbb{F}} \in \mathbb{R}^M$ by
\begin{equation*}
    (\hat{\mathbb{A}}_q)_{m,n} := \Btsq(\rbts{n}, \rbestts{m}) \qquad\text{and}\qquad
    \hat{\mathbb{F}}_m := \Fts(\rbestts{m}).
\end{equation*}
Then, with
$    \hat{\mathbb{A}}_\mu := \sum_{q=1}^{\Qa} \theta_q(\mu)\hat{\mathbb{A}}_q,$
solving~\eqref{eq:stage2_rom_riesz} is equivalent to solving the least-squares problem
\begin{equation}\label{eq:stage2_rom_in_matrices}
    c(\mu) := \argmin_{c \in \mathbb{R}^N} \norm{\hat{\mathbb{F}} - \hat{\mathbb{A}}_\mu\cdot c}^2,
\end{equation}
where
 $   \utsm = \sum_{n=1}^N c_n(\mu) \rbts{n}$.
In the same way we can evaluate the error bounds $\eta_{a,\mu}(\utsm)$ and $\eta_{1,\mu}(\uts)$ as
\begin{align*}
    \eta_{a,\mu}(\utsm) &= 
    \sqrt{5}\CPG^{-1}
    \norm{\hat{\mathbb{F}} - \hat{\mathbb{A}}_\mu\cdot c(\mu)},\\ 
    \eta_{1,\mu}(\utsm) &= 
    \sqrt{5}\CI\alpha^{-1/2}\CPG^{-1}
    \norm{\hat{\mathbb{F}} - \hat{\mathbb{A}}_\mu\cdot c(\mu)}.
\end{align*}
Since $M \leq \Qa N + 1$, the computational effort for assembling the least-squares system is of order
$\mathcal{O}(\Qa^2N^2)$. Solving the system requires $\mathcal{O}(\Qa N^3)$ operations, and
evaluating the estimators requires $\mathcal{O}(\Qa^2 N^2)$ operations.
In particular, the computational effort is completely independent from $h$ and $H$. 

We still need to show how the matrices $\hat{\mathbb{A}}_q$ and the vector $\hat{\mathbb{F}}$ can be computed after Stage~1
without using any data or operations associated with the fine mesh $\grid$.
To this end, we assume that $\Vtsrb \subseteq V_H \oplus \Vfrbkt[T_1] \oplus
\dots \oplus \Vfrbkt[T_{\Tlast}] \subset \Vts$.
By construction of the $\Wfrbkt$, we see that for such a $\Vtsrb$, $\Wtsrb$ is a linear subpsace of $V_H \oplus \Wfrbkt[T_1] \oplus
\dots \Wfrbkt[\Tlast]$.
Choose a basis $\rbts{n}$ of $\Vtsrb$ with coefficient vectors
\begin{equation*}
    \rbtscoeff{n} = (\rbtscoeff{n,H}, \rbtscoeff{n,T_1}, \dots, \rbtscoeff{n,\Tlast}) \in
    \mathbb{R}^{\dimVH} \oplus \mathbb{R}^{N_{T_1}} \oplus \dots \oplus \mathbb{R}^{N_{\Tlast}}
\end{equation*}
w.r.t.\ the finite element basis $\phi_i$ of $V_H$ and the reduced bases $\rbt{n}$ of $\Vfrbkt$.
Denote by $\mathbb{S}$ the $\dimVH \times \dimVH$ matrix of the Sobolev 1-inner product on $V_H$ given by
\begin{equation*}
    \mathbb{S}_{j,i} := \int_\Omega \nabla \phi_i \cdot \nabla \phi_j \,\mathrm{dx},
\end{equation*}
and use the $H^1$-orthonormal bases $\rbestt{m}$ to isometrically represent the
vectors~\eqref{eq:two_scale_estimator_generators} as coefficient vectors in the
direct sum Hilbert space $\underline{\mathfrak{W}}^{rb} := \mathbb{R}^{\dimVH} \oplus \mathbb{R}^{M_{T_1}} \oplus \dots
\mathbb{R}^{M_{\Tlast}}$ equipped with the $\mathbb{S}$-inner product in the first and with the Euclidean inner products in the
remaining components.
Checking the definitions of $\Btsq$, $\Fts$, $\hat{\mathbb{A}}^T_q$, $\hat{\mathbb{G}}^T_{q,j}$, $\Ktqo$ and $\Ktqrb$, we obain the
vectors
\begin{equation}\label{eq:two_scale_est_gen_in_coord_f}
    (\mathbb{S}^{-1}\cdot[F(\phi_i)]_i,0\dots,0),
\end{equation}
and
\begin{equation}\label{eq:two_scale_est_gen_in_coord_b}
    \begin{pmatrix}
        \mathbb{S}^{-1}\cdot\sumT \left(\Ktqo \cdot \rbtscoeff{n,H} - \Ktqrb \cdot \rbtscoeff{n,T}\right) \\
    \rho^{1/2}\hat{\mathbb{A}}^{T_1}_q\rbtscoeff{n,T_1}-
    \rho^{1/2}\sum_{j=1}^{J_{T_1}}\rbtscoeff{n,H,i_{T_1,j}}\hat{\mathbb{G}}^{T_1}_{q,j} \\
    \vdots\\
    \rho^{1/2}\hat{\mathbb{A}}^{\Tlast}_q\rbtscoeff{n,\Tlast}-
    \rho^{1/2}\sum_{j=1}^{J_{\Tlast}}\rbtscoeff{n,H,i_{\Tlast,j}}\hat{\mathbb{G}}^{\Tlast}_{q,j}
    \end{pmatrix},
\end{equation}
for $1 \leq n \leq N$.
After having computed~\eqref{eq:two_scale_est_gen_in_coord_f} and~\eqref{eq:two_scale_est_gen_in_coord_b}, we compute a
$\underline{\mathfrak{W}}^{rb}$-orthonormal basis $\rbesttscoeff{m}$ for these vectors with coefficients
\begin{equation*}
    \rbesttscoeff{m} = (\rbesttscoeff{m,H}, \rbesttscoeff{m,T_1}, \dots, \rbesttscoeff{m,\Tlast}),
\end{equation*}
such that 
\begin{equation*}
    \rbestts{m} :=
    (\sum_{i=1}^{\dimVH}\rbesttscoeff{m,H,i}\cdot\phi_i,
    \sum_{l=1}^{M_{T_1}}\rbesttscoeff{m,T_1,l}\cdot\rbestt{l},
    \dots,
    \sum_{l=1}^{M_{\Tlast}}\rbesttscoeff{m,\Tlast,l}\cdot\rbestt{l}
    )
\end{equation*}
is an $\Snorm{}$-orthonormal orthonormal basis for $\Wtsrb$.

Finally, following the definitions again, we see that $\hat{\mathbb{A}}_q$ and $\hat{\mathbb{F}}$ can be computed as
\begin{equation*}
    \hat{\mathbb{F}}_m = \mathcal{F}(\rbestts{m}) = \sum_{i=1}^{\dimVH}\rbesttscoeff{m,H,i}\cdot F(\phi_i),
\end{equation*}
and
\begin{align*}
    (\hat{\mathbb{A}}_q)_{m,n} 
    &= \Btsq(\rbts{n}, \rbestts{m}) \\
    &=
    \begin{multlined}[t]
        \sumT  \rbesttscoeff{m,H}^T \cdot (\Ktqo \cdot \rbtscoeff{n,H} - \Ktqrb \cdot \rbtscoeff{n,T}) \\
        + \rho^{1/2}\sumT\left( (\rbesttscoeff{m,T})^T\cdot\hat{\mathbb{A}}^T_q\cdot\rbtscoeff{n,T}
        - \sum_{j=1}^{J_{\Tlast}}(\rbesttscoeff{m,T})^T\cdot\hat{\mathbb{G}}^{T}_{q,j}\cdot\rbtscoeff{n,H,i_{T,j}}\right).
    \end{multlined}
\end{align*}

\subsubsection{Basis generation}\label{sec:stage_2_bas_gen}
To build the reduced Stage 2 space $\Vtsrb$, we follow the same methodology as in~\cref{sec:stage1_basisgen} and
use a greedy search procedure to iteratively extend $\Vtsrb$ until a given error tolerance $\varepsilon_2$ for the
MOR error estimate $\eta_{a,\mu}(\utsmrb)$ is reached.
However, in contrast to~\cref{sec:stage1_basisgen}, we will not use the full-order model~\eqref{eq:two_scale_PGLOD}, or
equivalently~\eqref{eq:PG_in_matrices}, to
compute solution snapshots, but rather its Stage 1 approximation. I.e., we solve
\begin{equation}\label{eq:stage_2_fom}
   \Kmurb \cdot \uhkmcoeff = \mathbb{F}
\end{equation}
to determine the $V_H$-component of the solution, followed by solving the Stage 1 corrector
ROMs~\eqref{eq:stage1_rom} for each $T \in \Gridh$ to determine the fine-scale components $\QQktmrb(\uhkm)$ of
the two-scale solution snapshot.
The full algorithm is given by Algorithm~\ref{alg:stage2_greedy}. 

Since we only extend $\Vtsrb$ with approximations of the true solution snapshots of the full-order model,
note that Algorithm~\ref{alg:stage2_greedy} is no longer a weak greedy algorithm in the sense of~\cite{BinevCohenEtAl2011}.
In particular note that the model reduction error is generally non-zero even for parameters $\mu^*$ for which the
corresponding Stage 1 snapshot has been added to $\Vtsrb$.
Thus, when $\varepsilon_1$ is chosen too large in comparison to $\varepsilon_2$, a single $\mu^*$ might be selected
twice, causing Algorithm~\ref{alg:stage2_greedy} to fail.
In such a case, the individual Stage 1 errors for each $T \in \Gridh$ can be estimated to further enrich the Stage 1
spaces for which the error is too large.
We will not discuss such an approach in more detail here and instead note that in practice it is feasible to choose 
$\varepsilon_1$ small enough to avoid such issues (cf.~\cref{sec:experiments}).
In particular, for sufficiently small $\varepsilon_1$ we can expect the convergence rates of a weak greedy algorithm
with exact solution snapshots to be preserved by Algorithm~\ref{alg:stage2_greedy} up to the given target tolerance
$\varepsilon_2$.

\begin{algorithm2e}
   \KwData{$\ParamsTrain$, $\varepsilon_2$}
   \KwResult{$\Vtsrb$}
   $\Vtsrb \leftarrow \{0\}$\;
   \While{$\max_{\mu\in\ParamsTrain} \eta_{a,\mu}(\utsmrb) > \varepsilon_2$}{
       $\mu^* \leftarrow \argmax_{\mu^* \in \ParamsTrain}\eta_{a,\mu}(\utsmrb)$\;
       $\uhkm[\mu^*] \leftarrow \text{solution of~\eqref{eq:stage_2_fom}}$\;
       $\Vtsrb \leftarrow \Span(\Vtsrb \cup \{(\uhkm[\mu^*], \QQktmrb[T_1][\mu^*](\uhkm[\mu^*]), \dots, \QQktmrb[\Tlast][\mu^*](\uhkm[\mu^*]))\})$\;
   }
   \caption{Weak greedy algorithm for the generation of $\Vtsrb$.}\label{alg:stage2_greedy}
\end{algorithm2e}

\section{Numerical experiments}\label{sec:experiments}
In this section we apply the above described method to two test cases and evaluate its efficiency.
For the first, smaller problem we mainly investigate the MOR error and
the performance of the certified error estimator.
The second, large-scale problem will be used to assess the computational speed-up achieved by the TSRBLOD.\@

In both cases, we use structured 2D-grids $\grid$, $\Gridh$ with quadrilateral elements on the domain $\Omega=[0,1]^2$.
We specify the number of elements by $n_h \times n_h$ and $n_H \times n_H$ respectively. 
We use the interpolation operator from \cite[Example 3.1]{LODbook}
and choose the oversampling parameter $k$ as the first integer to satisfy $k > |\log(H)|$.
For the evaluation of the error estimator $\eta_{a, \mu}$~\eqref{eq:stage2_est_a} we
approximate $\gamma_k \approx \alpha^{1/2}\CI^{-1}$ and assume $C_{\mathcal{I}_H} \approx 1$. For quadrilateral elements, the overlapping constant can be explicitly computed
by $\CO = (2k + 1)^2$.
Moreover, we approximate $\alpha$ and the contrast $\kappa$ by replacing $\mathcal{P}$ in \cref{eq:ellipticity} by the training set
$\ParamsTrain$.
Thus, while it is not guaranteed that our approximations yield strict upper bounds on the model order reduction errror, the decay
rate and choice of snapshot parameters will not be affected.

We use an \texttt{MPI} distributed implementation to benefit from parallelization of the localized corrector problems,
which gives significant speed-ups for all considered methods.
All our computations have been performed on an HPC cluster with 1024 parallel processes.

\paragraph{Petrov--Galerkin variant of the RBLOD method}
In order to compare our method to the RBLOD approach introduced in~\cite{RBLOD},
we have implemented a corresponding version of the RBLOD that is applicable to our LOD formulation.
In particular, we use the interpolation operator from \cite[Example 3.1]{LODbook} and Petrov-Galerkin
projection~\eqref{eq:PG_} in contrast to the Cl\'ement interpolation and Galerkin projection used in~\cite{RBLOD}.
Further, in~\cite{RBLOD} individual ROMs for the local correctors $\QQktm(\phi_i)$ are constructed, where the snapshot parameters are chosen identically among all $\QQktm[T^\prime](\phi_i)$, $T^\prime \subset \operatorname{supp} \psi_z$.
In the RBLOD variant implemented by us, we independently train the corrector ROMS for each $T \in \Gridh$ as is done in Stage~1 of the TSRBLOD.\@
We note that for both RBLOD variants, the number of ROMs constructed for each coarse element $T$ is equal to the number of coarse-mesh basis functions $\phi_i$ supported on this element, whereas Stage 1 of the TSRBLOD constructs a single ROM to approximate all local correctors.

\paragraph{Error measures}
To quantify the accuracy of the TSRBLOD, we use a validation set $\ParamsVer \subset \mathcal{P}$ of $10$ random
parameters and compute the maximum relative approximation errors w.r.t.\ the PG--LOD
solutions for this set, i.e.,
\begin{align*}
e^{*, \text{rel}}_{\text{LOD}} := \max_{\mu\in\ParamsVer} \frac{\| \uhkm - \tilde{u}_{\mu} \|_{*} }{\|\uhkm\|_{*}}, 
\end{align*}
where $\tilde{u}_{\mu}$ denotes the coarse-scale component of either the
RBLOD or TSRBLOD solution, $\uhkm$ denotes the coarse component of the PG--LOD solution~\eqref{eq:PG_}, and $^*$ stands for either the $H^1$- or $L^2$-norm. With the solution of the FEM approximation $u_{h,\mu}$ w.r.t.~$\grid$, we let
\begin{align*}
e^{L^2, \text{rel}}_{\text{FEM}} &:=  \max_{\mu\in\ParamsVer} \frac{\| u_{h,\mu} - \tilde{u}_{\mu}\|_{L^2} }{\|u_{h,\mu}\|_{L^2}}, \qquad
 &e^{L^2, \text{rel}}_{\text{LOD-FEM}} &:=  \max_{\mu\in\ParamsVer} \frac{\| u_{h,\mu} - \uhkm \|_{L^2} }{\|u_{h,\mu}\|_{L^2}}.
\end{align*}

\paragraph{Time measures}
For an analysis of the computational wall times, we consider the following quantities:
\begin{itemize}[leftmargin=*]
    \item $t^\text{offline}_{1,\text{av}}(T)$: Average (arithmetic mean) time for creating a ROM for the localized corrector problem(s)
        corresponding to a single coarse element $T \in \Gridh$ as discussed in \cref{sec:stage1}.
	For the RBLOD , this involves all individual corrector problems for the four basis functions that are supported on
    $T$.
	\item $t^\text{offline}_{1}$: Total time required to build the corrector ROMs. In case of full parallelization, this
        is equal to the maximum of the Stage 1 times over all $T \in \Gridh$.
	\item $t^\text{offline}_2$: Time for creating the Stage 2 TSRBLOD ROM as discussed in \cref{sec:stage_2_red}.
	\item $t^{\text{offline}}$: Total offline time for building the final reduced model. For the RBLOD this is equal to $t^\text{offline}_{1}$. For the TSRBLOD this additionally includes $t^\text{offline}_2$.	
	\item $t^{\text{online}}$: Average time for solving the obtained ROM for a single new $\mu \in \ParamsVer$.
        In case of the RBLOD, the reduced corrector problems are solved sequentially on a single compute node.
	\item $t^{\text{LOD}}$: Average time needed to compute the PG--LOD with parallelization of the corrector problems
        for a single new $\mu \in \ParamsVer$.
\end{itemize}

\subsection{Test case 1}

The first test case is taken from~\cite[Section 4.1]{RBLOD} and has a 
one-dimensional parameter space $\Params:= [0,5]$. The coefficient $A_\mu$ is visualized in~\cref{mp1:pics}, and we set $f \equiv 1$.
For the sake of brevity, we refer to~\cite{RBLOD} for an exact definition of $A_\mu$.
In order to resolve the microstructure of the problem,
we choose $n_h = 2^{8}$ which results in $65,536$ fine-scale elements. The approximated maximum contrast of $A_\mu$ is $\kappa \approx 13$.

\begin{figure}[h]
	\centering
	\includegraphics[width=1.6in]{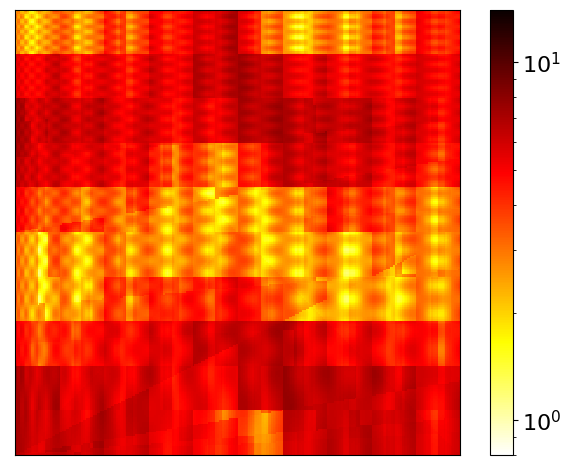}
	\includegraphics[width=1.6in]{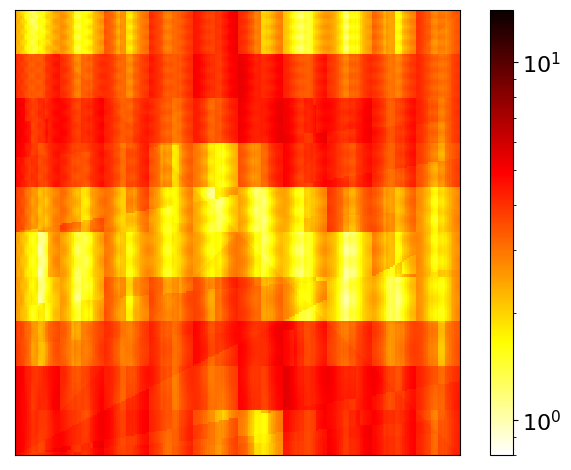}
	\includegraphics[width=1.6in]{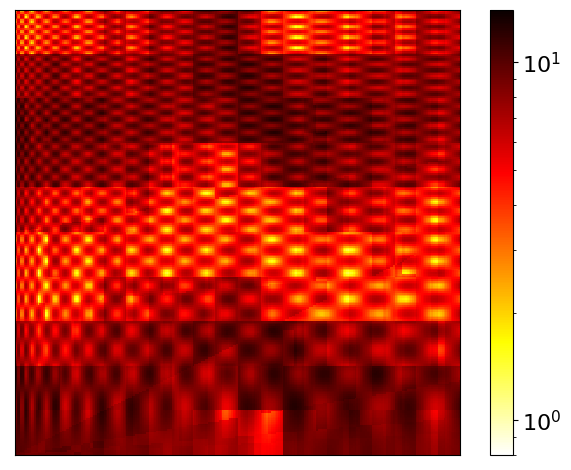}
	\captionof{figure}{\footnotesize{Diffusion coefficient $A_\mu$ for $\mu_1 = 1.8727$ (left), $\mu_2 = 2.904$ (middle), and $\mu_3 = 4.7536$ (right) for test case 1.}}
	\label{mp1:pics}
	\centering
\end{figure}

\begin{table}
	\footnotesize
	\centering
	{\def\arraystretch{1.7}\tabcolsep=3pt
		\begin{tabular}{|l|c|c|c|c|c|c|}
			mesh size $n_H$& \multicolumn{2}{c|}{$2^{3}$} & \multicolumn{2}{c|}{$2^{4}$} &\multicolumn{2}{c|}{$2^{5}$} \\ \hline
			method  &\hphantom{T}RBLOD\hphantom{T} & TSRBLOD&\hphantom{T}RBLOD\hphantom{T} & TSRBLOD 
			&\hphantom{T}RBLOD\hphantom{T} & TSRBLOD \\ \hline  \hline
			$t^\text{offline}_{1, \text{av}}(T)$ 	&	41		&  	61		&  	39		& 	61		&   33		&	55		\\
			$t^\text{offline}_{1}$   			 	&	71		&  	106		&  	67		& 	102		&   63 		&	98		\\
			$t^\text{offline}_2$   					&	- 		& 	8 		&  	-		& 	56  	&	- 		&	472	\\
			$t^\text{offline}$    					&	71		& 	114 	&  	67		& 	158		&   63		&	570	\\ \hline
			cum. size St.1 							&	2346	&  	1670	&	8718	& 	6134	&	31810 	&	22189	\\
			av. size St.1           				&	9.16	&  	26.09	& 	8.51	& 	23.96	&	7.77 	&	21.67	\\
			size St.2 								&	-		&  	8		& 	-		& 		9	&		 	&	9		\\
			\hline \hline
			$t^{\text{LOD}}$ 
			& 	\multicolumn{2}{c|}{0.69} 	
			& 	\multicolumn{2}{c|}{0.49}	
			& 	\multicolumn{2}{c|}{0.90}	 \\ \hline
			$t^\text{online}$ 		 				&	0.0610	&  	0.0003	&  	0.2272	& 	0.0003	&   1.0462 	&	0.0003	 \\ \hline \hline
			speed-up LOD 	 						&	11		&  	2506	& 	2 		& 	1536 	&   1		& 	2714	 \\ \hline \hline
			$e^{H^1, \text{rel}}_{\text{LOD}}$ 	 	&	1.97e-5	&	7.30e-4	&	5.08e-5 &	2.94e-4 &	1.11e-4 &	4.21e-4	 \\
			$e^{L^2, \text{rel}}_{\text{LOD}}$ 	 	&	4.89e-6 &	2.71e-4 &	6.77e-6 &	1.03e-4 &	7.70e-6	&	1.32e-4  \\
			\hline \hline
			$e^{L^2, \text{rel}}_{\text{FEM}}$ 		&	2.46e-2 &	2.46e-2	&	9.05e-3	&	9.05e-3 &	3.98e-3	&	3.98e-3 \\ \hline
			$e^{L^2, \text{rel}}_{\text{LOD-FEM}}$	 
			& 	\multicolumn{2}{c|}{2.46e-2} 
			& 	\multicolumn{2}{c|}{9.05e-3} 
			& 	\multicolumn{2}{c|}{3.98e-3}\\ \hline
	\end{tabular}}
	\captionsetup{width=\textwidth}
	\caption{\footnotesize{%
			Performance, ROM sizes and accuracy of the methods for test case 1 with tolerances $\varepsilon_1 = 0.001$ and $\varepsilon_2 = 0.01$ and varying coarse-mesh sizes. All times are given in seconds.	\label{Tab:mp1:1}}}
\end{table}

\subsubsection{Performance and error comparison}
We vary the coarse mesh size, $n_H \in \{2^{3},2^{4},2^{5}\}$, and show results for
$\varepsilon_1 = 0.001$, $\varepsilon_2=0.01$ and a training set $\ParamsTrain \subset \Params$ of $50$ equidistant
parameters in \Cref{Tab:mp1:1}.

We observe that this choice of tolerances clearly suffices for both the RBLOD and TSRBLOD to match the error of the PG--LOD
solution w.r.t.\ the FEM reference solution.
For all coarse mesh sizes, the TSRBLOD produces a ROM with 8 or 9 basis functions without loosing accuracy.
The offline phase of Stage~1 for the TSRBLOD is longer and yields larger ROMS per element than
the RBLOD.\@
However, for the RBLOD four ROMS per element are required s.t.\ the total number of basis vectors actually is smaller for
the TSRBLOD.\@
The offline time for Stage~2 increases for larger $n_H$ due to the increasing complexity of assembling and solving the
coarse-scale problem, even with an RB approximation of the corrector problems.
However, this only effects the offline phase for the TSRBLOD, and the online times stay at a constant level since the
dimension of the Stage~2 ROM is largely unaffected by the number of coarse elements.
Considering the speed-ups achieved by both methods, it can clearly be seen that the RBLOD has just a slight
benefit over the (parallelized) PG--LOD, whereas the TSRBLOD shows a significant speed-up for all $n_H$ without
requiring any parallelization.

\subsubsection{Error and estimator decay}

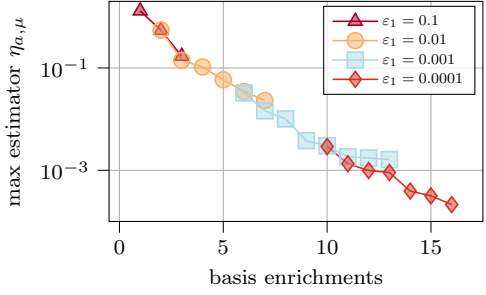
\begin{figure}[t]
	\footnotesize
	\centering
	\begin{tikzpicture}
	\definecolor{color0}{rgb}{0.65,0,0.15}
	\definecolor{color1}{rgb}{0.84,0.19,0.15}
	\definecolor{color2}{rgb}{0.96,0.43,0.26}
	\definecolor{color3}{rgb}{0.99,0.68,0.38}
	\definecolor{color4}{rgb}{1,0.88,0.56}
	\definecolor{color5}{rgb}{0.67,0.85,0.91}
	\begin{axis}[
	name=left,
	anchor=west,
	width=6.5cm,
	height=4.5cm,
	log basis y={10},
	tick align=outside,
	tick pos=left,
	legend style={nodes={scale=0.7}, fill opacity=0.8, draw opacity=1, text opacity=1, 
	},
	x grid style={white!69.0196078431373!black},
	xlabel={basis enrichments},
	ylabel={max estimator $\eta_{a, \mu}$},
	xmajorgrids,
	xtick style={color=black},
	y grid style={white!69.0196078431373!black},
	ymajorgrids,
	ymode=log,
	ymin=1e-04, ymax=2,
	ytick style={color=black},
	ytick pos=left,
	]
	\addplot [semithick, color0, mark=triangle*, mark size=3, mark options={solid, fill opacity=0.5}]
	table {%
		1 1.302208490599 
		2 0.542298633478 
		3 0.169259772169
	};
	\addlegendentry{$\varepsilon_1 = 0.1\hphantom{000}$}
	\addplot [semithick, color3, mark=*, mark size=3, mark options={solid, fill opacity=0.5}]
	table {%
		2 0.548045241003 
		3 0.140901484479 
		4 0.104570132829 
		5 0.058864107218 
		6 0.034589798089 
		7 0.023123372984
	};
	\addlegendentry{$\varepsilon_1 = 0.01\hphantom{00}$}
	\addplot [semithick, color5, mark=square*, mark size=3, mark options={solid, fill opacity=0.5}]
	table {%
		6 0.032599625814 
		7 0.014522966903 
		8 0.010109952521 
		9 0.003802794378 
		10 0.003003471958 
		11 0.001834614254 
		12 0.001739345745 
		13 0.001618788041
	};
	\addlegendentry{$\varepsilon_1 = 0.001\hphantom{0}$}
	\addplot [semithick, color1, mark=diamond*, mark size=3, mark options={solid, fill opacity=0.5}]
	table {%
		10 0.002914232833 
		11 0.001344315689 
		12 0.000991445067 
		13 0.000905171861 
		14 0.000392978852 
		15 0.000316746807 
		16 0.000214242687
	};
	\addlegendentry{$\varepsilon_1 = 0.0001$}
	\end{axis}
	\end{tikzpicture}
	\captionsetup{width=\textwidth}
	\caption{\footnotesize{Evolution of the maximum estimated error $\eta_{a, \mu}$, $\mu \in \ParamsTrain$ for different Stage~1 tolerances $\varepsilon_1$ during the greedy algorithm of Stage 2 ($n_H = 2^4$). The greedy algorithm has been continued until the enrichment failed. The first values for smaller tolerances are left out to improve readability.}}
	\label{Fig:estimator_study}
\end{figure}

\begin{figure}[t]
	\footnotesize
	\centering
	\begin{tikzpicture}
	\definecolor{color0}{rgb}{0.65,0,0.15}
	\definecolor{color1}{rgb}{0.84,0.19,0.15}
	\definecolor{color2}{rgb}{0.96,0.43,0.26}
	\definecolor{color3}{rgb}{0.99,0.68,0.38}
	\definecolor{color4}{rgb}{1,0.88,0.56}
	\definecolor{color5}{rgb}{0.67,0.85,0.91}
	\begin{axis}[
	name=left,
	anchor=west,
	width=6.5cm,
	height=4.5cm,
	log basis y={10},
	legend cell align={left},
	legend style={nodes={scale=0.7}, fill opacity=0.8, draw opacity=1, text opacity=1, at=(left.west), anchor=north east, xshift=-0.2cm, yshift=0.7cm, draw=white!80!black},
	x grid style={white!69.0196078431373!black},
	xlabel={basis enrichments},
	xmajorgrids,
	xtick style={color=black},
	y grid style={white!69.0196078431373!black},
	ymajorgrids,
	ymode=log,
	ymin=5e-08, ymax=10,
	ytick style={color=black},
	ytick pos=right
	]
	\addplot [semithick, color0, mark=x, mark size=4, mark options={solid, fill opacity=0.5}]
	table {%
		1 0.20583663158433416
		2 0.11887146368778598
		3 0.029783870160454723
		4 0.01489840725378891
		5 0.014067080101043957
		6 0.007252894748508427
		7 0.003190748490468686
		8 0.0016313851696358674
		9 0.0007856473067503389
		10 0.0007520128559338216
		11 0.00038889201711461214
		12 0.00038219438646071996
		13 0.00034000389283784436
	};
	\addplot [semithick, black, mark=diamond*, mark size=3, mark options={solid, fill opacity=0.2}]
	table {%
		1 0.013304715193246983
		2 0.009712437622371539
		3 0.002816468173892014
		4 0.0011318376672740094
		5 0.0010401116607180129
		6 0.0006290529115014175
		7 0.00014645370454404016
		8 0.00008648122221639462
		9 0.000045722930894291935
		10 0.00004620880598336765
		11 0.000021310199120887216
		12 0.00001932049892755662
		13 0.00001260466561162649
	};
	\addplot [semithick, color1, mark=square*, mark size=2, mark options={solid, fill opacity=0.2}]
	table {%
		1 0.01167264689554406
		2 0.008934578133502196
		3 0.002677895817283881
		4 0.000982896580291272
		5 0.0008893416590385863
		6 0.0005798216284228025
		7 6.493345512553352e-05
		8 5.704378512652952e-05
		9 3.59582145370613e-05
		10 3.677364143091256e-05
		11 1.7109809385720184e-05
		12 1.2277935969088866e-05
		13 2.4157504039053084e-06
	};
	\addplot [semithick, black, mark=square*, opacity=0.4, mark size=1.5, mark options={solid, fill opacity=0.5}]
	table {%
		1 0.20556287671660609
		2 0.11853521920869325
		3 0.029663239804960596
		4 0.014895622925657302
		5 0.014041884984912538
		6 0.00722968112105605
		7 0.0031902180521013682
		8 0.0016303875546281278
		9 0.0007852344218980461
		10 0.0007511131970519845
		11 0.00038867189820279414
		12 0.00038202110052247513
		13 0.00034000203935049633
	};
	
	\addplot [semithick, color0, mark=*, mark size=3, dashed, mark options={solid, fill opacity=0.5}]
	table {%
		1 1.309575235
		2 0.548060703
		3 0.141635678
		4 0.104847183
		5 0.059237218
		6 0.032599626
		7 0.014522967
		8 0.010109953
		9 0.003802794
		10 0.003003472
		11 0.001834614
		12 0.001739346
		13 0.001618788
	};
	
	\addplot [semithick, color5, mark=square*, mark size=3, mark options={solid, fill opacity=0.5}]
	table {%
		1 7.12906094171107
		2 7.3606022714764565
		3 6.969506743881207
		4 7.084436819956133
		5 6.662971068069086
		6 6.681572675067291
		7 6.359053731851468
		8 6.567192608633759
		9 6.530567472242883
		10 6.6733601290328055
		11 6.676811460567216
		12 6.59853821648667
		13 6.544061489519691
	};
	\addplot [semithick, color3, mark=triangle*, dashed, mark size=3, mark options={solid, fill opacity=0.5}]
	table {
		1 0.064086520
		2 0.033603723
		3 0.012175896
		4 0.004836665
		5 0.004318154
		6 0.002758986
		7 0.000401127
		8 0.000339618
		9 0.000153555
		10 0.000157558
		11 0.000087057
		12 0.000054408
		13 0.000013076
	};
	\end{axis}
	\begin{axis}[
	name=right,
	anchor=west,
	at=(left.east),
	xshift=0.9cm,
	width=6.5cm,
	height=4.5cm,
	log basis y={10},
	legend cell align={left},
	legend style={nodes={scale=0.7}, fill opacity=0.8, draw opacity=1, text opacity=1, at=(left.east), anchor=north east, xshift=-11cm, yshift=0.7cm, draw=white!80!black},
	tick align=outside,
	tick pos=left,
	x grid style={white!69.0196078431373!black},
	xlabel={basis enrichments},
	xmajorgrids,
	ymode=log,
	ymin=5e-08, ymax=10,	
	xtick style={color=black},
	y grid style={white!69.0196078431373!black},
	ymajorgrids,
	yticklabels={,,},
	ytick pos=left,
	]
	\addplot [semithick, color0, mark=x, mark size=4, mark options={solid, fill opacity=0.5}]
	table {
		1 0.20583663158433418
		2 0.11694206395586734
		3 0.02973104282919254
		4 0.0187861354798649
		5 0.013858503372350388
		6 0.012416870025239479
		7 0.0026686216441154298
		8 0.002635502193860233
		9 0.0007400980127743919
		10 0.0005443247667907771
		11 0.0004620402045593204
		12 0.0004127701959231501
		13 0.00034019560933311744
		14 0.00034018968282338157
		15 0.00034006575626278755
		16 0.00034005631077660875
		17 0.00034000677770389305
		18 0.00034000614761857314
		19 0.0003400047029756307
		20 0.00034000469392983946
		21 0.0003400046813423982
		22 0.0003400046754256699
		23 0.0003400046507017614
		24 0.00034000451666054874
		25 0.00034000442076877424
	};
	\addlegendentry{two-scale error}
	\addplot [semithick, black, mark=diamond*, mark size=3, mark options={solid, fill opacity=0.2}]
	table {
		1 0.013304715193246948
		2 0.009511015595946179
		3 0.0028113477767885784
		4 0.0010418776425366626
		5 0.000695950849759363
		6 0.0007067760148770652
		7 0.0001288287554918048
		8 0.00012289411897222592
		9 4.5151058847703486e-05
		10 2.9558799782869537e-05
		11 2.237760999975541e-05
		12 1.9830356206882944e-05
		13 1.2628238317123184e-05
		14 1.2625312273148166e-05
		15 1.2624222828842822e-05
		16 1.2622672901671814e-05
		17 1.2608270221566213e-05
		18 1.2607216790604492e-05
		19 1.2605206742112442e-05
		20 1.2605264311413092e-05
		21 1.2605117349898283e-05
		22 1.260517373978866e-05
		23 1.2605059145931392e-05
		24 1.2604687600214781e-05
		25 1.2605079504893575e-05		
	};
	\addlegendentry{LOD error}
	\addplot [semithick, color1, mark=square*, mark size=2, mark options={solid, fill opacity=0.2}]
	table {
		1 0.01167264689554405
		2 0.008727740382171471
		3 0.0026728177591846536
		4 0.0007560533227358547
		5 0.0005811156126797668
		6 0.0004809482874223695
		7 7.457221511493489e-05
		8 5.032735352846313e-05
		9 3.5721462437374986e-05
		10 2.3204288829405183e-05
		11 1.492326673635207e-05
		12 1.0182680403255702e-05
		13 2.0210071303287623e-06
		14 2.027257098624611e-06
		15 1.6766888163742802e-06
		16 1.1714618468972557e-06
		17 1.4449245167868556e-06
		18 1.12423583570912e-06
		19 1.1229860536982722e-06
		20 1.1230851612871708e-06
		21 1.1228418969885156e-06
		22 1.1229347412743485e-06
		23 1.122478451417174e-06
		24 1.1225804318428979e-06
		25 1.122916984414993e-06
	};
	\addlegendentry{$V_H$-error}
	\addplot [semithick, black, mark=square*, mark size=1.5, opacity=0.4, mark options={solid, fill opacity=0.5}]
	table {
		1 0.20556287671660609
		2 0.11661592031142051
		3 0.02961065607070313
		4 0.01877402382243087
		5 0.013853923304489067
		6 0.012407552134426684
		7 0.002667964814102145
		8 0.0026350286117542217
		9 0.0007392354466838959
		10 0.0005439894059769174
		11 0.00046195181513729793
		12 0.0004126528759191995
		13 0.0003401935686116118
		14 0.00034018763966547634
		15 0.00034006375702974264
		16 0.0003400542929829571
		17 0.00034000491152417437
		18 0.00034000428896149035
		19 0.0003400028484408183
		20 0.00034000283906762354
		21 0.00034000282728356853
		22 0.00034000282106018176
		23 0.0003400027978428308
		24 0.00034000266346419613
		25 0.00034000256646054275
	};
	\addlegendentry{corrector error}
	
	\addplot [semithick, color0, mark=*, mark size=3, dashed, mark options={solid, fill opacity=0.5}]
	table {
		1 1.309575235
		2 0.538279469
		3 0.143527566
		4 0.129353986
		5 0.096122047
		6 0.091427914
		7 0.013125357
		8 0.012644495
		9 0.002991285
		10 0.002984935
		11 0.002913277
		12 0.002555585
		13 0.001619900
		14 0.001619876
		15 0.001619153
		16 0.001619044
		17 0.001618832
		18 0.001618806
		19 0.001618797
		20 0.001618799
		21 0.001618796
		22 0.001618799
		23 0.001618800
		24 0.001618798
		25 0.001618796
	};
	\addlegendentry{$\eta_{a, \mu}$}		
	
	\addplot [semithick, color3, mark=triangle*, mark size=3, dashed, mark options={solid, fill opacity=0.5}]
	table {
		1 0.064086520
		2 0.032843009
		3 0.012154224
		4 0.005207466
		5 0.002802488
		6 0.003292533
		7 0.000361567
		8 0.000210742
		9 0.000153361
		10 0.000119555
		11 0.000072419
		12 0.000062428
		13 0.000008881
		14 0.000008243
		15 0.000007220
		16 0.000004308
		17 0.000007923
		18 0.000001845
		19 0.000001591
		20 0.000000897
		21 0.000000279
		22 0.000000216
		23 0.000000209
		24 0.000000103
		25 0.000000080
	};
	\addlegendentry{$\eta_{a, \mu}$, $\rho = 0$}
	\addplot [semithick, color5, mark=square*, mark size=3, mark options={solid, fill opacity=0.5}]
	table {
		1 7.129060941711069
		2 7.3656688584410945
		3 6.9824580069300275
		4 6.885609149899999
		5 7.333056748247623
		6 7.3632013268725025
		7 6.471238841821415
		8 6.566019043110956
		9 6.626142811888608
		10 6.624813706577451
		11 6.658917543462653
		12 6.556560885917927
		13 6.475931924186771
		14 6.483625325095644
		15 6.4830024301230464
		16 6.477027190303722
		17 6.482799619476211
		18 6.481894302317812
		19 6.481905020327957
		20 6.478519232309123
		21 6.4822406600512705
		22 6.482246722522054
		23 6.482690110890613
		24 6.4811915922350805
		25 6.481407203351824
	};
	\addlegendentry{effectivity $\eta_{a, \mu}$}
	
	\end{axis}
	\node[anchor=south east, yshift=4pt, xshift=-20pt] at (left.north east) {(A) train with $\eta_{a, \mu}$};
	\node[anchor=south west, yshift=4pt, xshift=20pt] at (right.north west) {(B) train with $\eta_{a, \mu}$, $\rho = 0$};
	\end{tikzpicture}
	\captionsetup{width=\textwidth}
	\caption{\footnotesize{Comparison of the Stage~2 training error decay using different error estimators ($\varepsilon_1=0.001$). The greedy algorithm is continued until the enrichment fails or $25$ enrichments are reached.
    Depicted is the maximum value of the full two-scale error $\Anorm{\utsm - \uts^{rb}_\mu}$, the LOD error given by
    $\Anorm{\utsm - \uts^{rb}_\mu}$ with $\rho=0$,
    the $V_H$-error $\anorm{\uhkm - \uhkm^\text{rb}}$,
	the fine-scale corrector error $\rho^{1/2} \cdot (\sumT \anorm{\QQktm(\uhkm^\text{rb}) -  u^\text{rb,f}_T}^2)^{1/2}$,
    the estimator $\eta_{a, \mu}$, its effectivity $\eta_{a, \mu}/\Anorm{\utsm - \uts^{rb}_\mu}$
    and the part of $\eta_{a, \mu}$ corresponding to the LOD residual (obtained by setting $\rho=0$).
    The maximum is computed over the training set $\ParamsTrain$.
    The error estimator used in the greedy algorithm is either (A) the two-scale error estimator $\eta_{a, \mu}$ or (B)
    only its LOD-residual part ($\rho=0$).}}
	\label{Fig:estimator_study_2}
\end{figure}
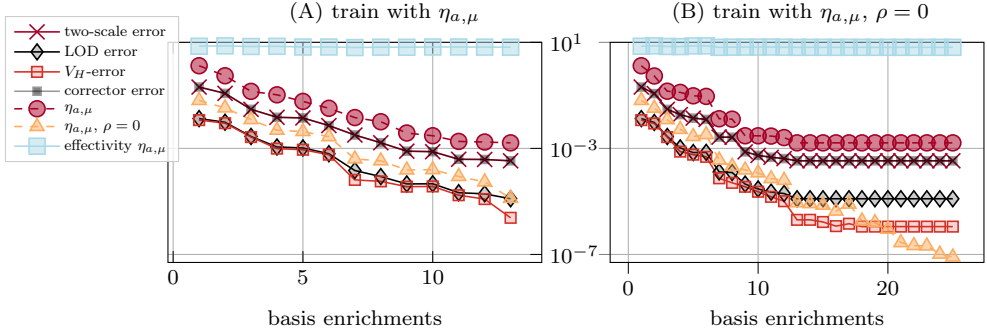

Next we study the influence of the Stage~1 tolerance $\varepsilon_1$ on the training of the Stage 2 ROM.
To this end, we fix the number of coarse mesh elements $n_H = 2^4$ and
depict in \cref{Fig:estimator_study} for different $\varepsilon_1$ the maximum estimated training error w.r.t.\ the number of basis
functions of the Stage 2 ROM.
For all $\varepsilon_1$, the Stage~2 training was continued until enrichment failed due to repeated selection of the
same snapshot parameter (cf.\ \cref{sec:stage_2_bas_gen}).
As expected, a sufficiently small $\varepsilon_1$ is required to achive small errors for the Stage~2 ROM.
Note, however, that choosing a smaller $\varepsilon_1$ for the same $\varepsilon_2$ only affects the offline time of the
Stage 2 training, but not the efficiency of the resulting Stage~2 ROM.

In \cref{Fig:estimator_study_2}(A), we study in more detail how the two-scale error $\Anorm{\utsm - \utsm^\text{rb}}$ and its estimator $\eta_{a, \mu}$ are affected by the coarse- and fine-scale errors in the two-scale system.
We observe that the greedy algorithm aborts when the error in the fine-scale correctors stagnates
at the lower bound determined by the fixed Stage~1 ROMs which are used to generate the Stage~2 solution snapshots.\@
While the LOD error ($\Anorm{}$ with $\rho = 0$) and the corresponding residual norms ($\eta_{a,\mu}$ with $\rho = 0$)
decrease over all iterations, the corrector residuals dominate $\eta_{a,\mu}$, finally causing the same snapshot parameter
to be selected twice.

To verify that the enrichment procedure should, indeed, be stopped at this point,
we perform another experiment where we neglect the corrector residuals and use $\eta_{a, \mu}$ with $\rho = 0$ as the error surrogate in the greedy algorithm.
This corresponds to treating the RBLOD coarse system as the FOM w.r.t.\ which the MOR error is estimated.
The result is visualized in \cref{Fig:estimator_study_2}(B).
We see that, while the estimator rapidly decays over all 25 enrichments, both the two-scale and LOD errors do not decrease any
further.
This is expected as both error measures involve the error in the fine-scale correctors.
However, also the $V_H$-error $\anorm{\uhkm - \uhkm^\text{rb}}$ stagnates after one further iteration,
and the estimator eventually underestimates all three errors.
This underlines that the fine-scale corrector errors need to taken into account, even if one is only interested in a coarse-scale approximation in $V_H$.

\subsection{Test case 2}
We now consider a more complex test case with significantly larger fine-scale mesh $\grid$.
The three-dimensional parameter space is given by $\Params := [1,5]^3$,
and we let $A_\mu:=\sum_{q=1}^{3}\mu_q A_q$, where for a representative patch of four coarse-mesh elements the randomly generated functions $A_q$ are given according to~\cref{mp2:pics}, and we again set $f \equiv 1$.
The exact definition of the functions $A_q$ can be found in the accompanying code.
The approximate maximum contrast is $\kappa \approx 16$.

To fully resolve the microstructure on all $4096$ ($n_H = 2^{6}$) coarse elements, we need to choose
$n_h = 2^{13}$ which results in about $67.1$ million degrees of freedom.
The local corrector problems have a size of roughly $1.3$ million degrees of freedom.
With 1024 available parallel processes, each process must train ROMs for $4$ coarse elements.

\begin{figure}
	\begin{center}
		\begin{subfigure}[b]{0.38\textwidth}
			\includegraphics[width=\textwidth]{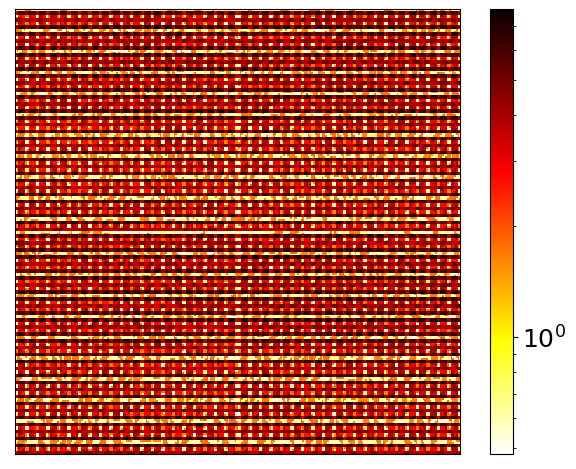}
		\end{subfigure}
	\end{center}
	\begin{subfigure}[b]{0.26\textwidth}
		\includegraphics[width=\textwidth]{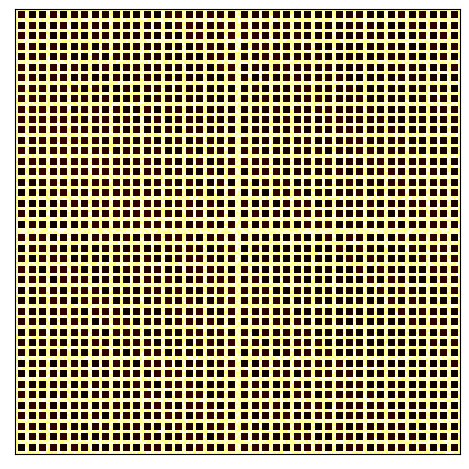}
	\end{subfigure}
	\begin{subfigure}[b]{0.26\textwidth}
		\includegraphics[width=\textwidth]{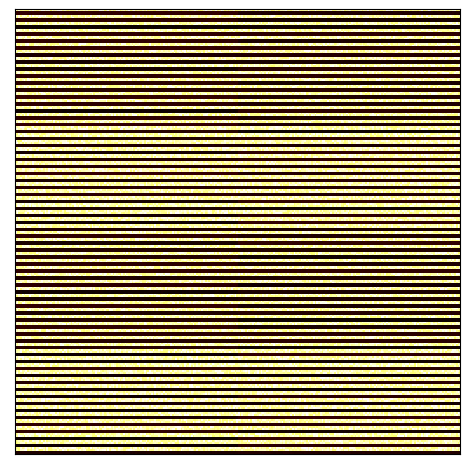}
	\end{subfigure}
	\begin{subfigure}[b]{0.328\textwidth}
		\includegraphics[width=\textwidth]{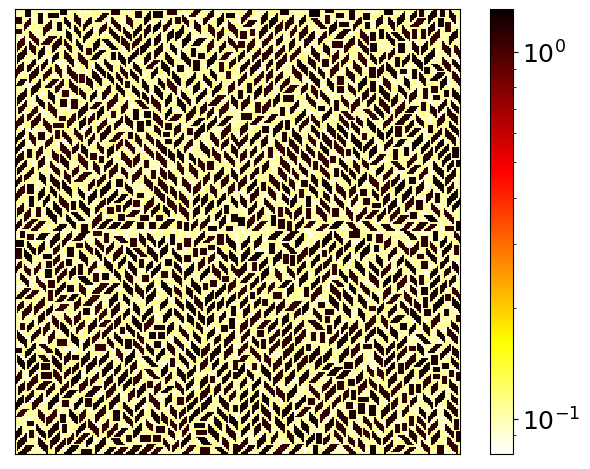}
	\end{subfigure}
	\centering
	\caption{\footnotesize{Coefficient $A_\mu$ on $4$ of $4096$ coarse elements for $\mu=(1,2,3)^T$ (top center) and $A_q$ for all $q=1,2,3$ (bottom from left to right).
			All coefficients $A_q$ are subjected to normally distributed noise in the interval $[1,1.2]$ for the particles (black) and in the interval $[0.03,0.11]$ for the background (yellow). To ensure reproducibility, for every coarse-mesh element $T$, we choose the random seed for the noise and for the distribution of the randomly shaped particles in $A_3$ as the global index of $T$ in $\Gridh$.}}
	\label{mp2:pics}
\end{figure}

We choose a training set $\ParamsTrain \subset \Params$ of $4^3$ equidistant parameters
and show results for $\varepsilon_1 = 0.01$ and $\varepsilon_2=0.02$ in \cref{Tab:mp2:1}. 
Again, the TSRBLOD shows high online efficiency with speed-ups of up to $10^6$, clearly outperforming the RBLOD.
We note that the reported online time of $t^{\text{online}}=4.39\,\mathrm{s}$ for the RBLOD roughly splits into $2.2\,\mathrm{s}$ required for sequentially solving the reduced corrector problems (Step~1 in \cref{sec:sub_comp_comp}), $1.2\,\mathrm{s}$ for assembling the coarse system (Step~2) and
$1\,\mathrm{s}$ for solving the coarse system (Step~3).
In particular, the further speed-up that could be achieved for the RBLOD by parallelizing the corrector problems is bounded by a factor of approximately 2.

Also, the storage requirements are noteworthy:
the reduced data required for evaluating the TSRBLOD ROM and $\eta_{a,\mu}$ is only 28~KB in
size, whereas the RBLOD requires 409~MB.

\begin{table}
	\footnotesize
	\centering
	{\def\arraystretch{1.7}\tabcolsep=3pt
		\begin{tabular}{|l|c|c|}
			method  &\hphantom{T}RBLOD\hphantom{T} & TSRBLOD 
			\\ \hline  \hline
			$t^\text{offline}_{1}(T)$   &   10278	& 	11289	\\
			$t^\text{offline}_{1}$   	&   49436 	& 	54837	\\
			$t^\text{offline}_2$   		& 	- 		&	9206 	\\
			$t^\text{offline}$    		& 	49436 	&	64043	\\ \hline
			cum. size St.1 				&	278528	&	193289 	\\
			av. size St.1           	&	17.00  	&	47.19	\\
			size St.2 					& 	- 		& 	16 		\\
			\hline \hline
			$t^{\text{LOD}}$ &  	\multicolumn{2}{c|}{515} \\ \hline
			$t^\text{online}$ 		 	            & 	4.39 	& 	0.0005	\\ \hline \hline
			speed-up w.r.t LOD 	 		            & 	117  	&	9.57e5	\\ \hline \hline
			$e^{H^1, \text{rel}}_{\text{LOD}}$ 	 	&	1.95e-5 &	4.43e-4 \\
			$e^{L^2, \text{rel}}_{\text{LOD}}$ 	 	&	2.36e-5 &	4.49e-4 \\ 
			\hline
	\end{tabular}}
	\captionsetup{width=\textwidth}
	\caption{\footnotesize{%
			Performance, ROM sizes and accuracy of the methods for test case 2 with $\varepsilon_1 = 0.01$ and $\varepsilon_2=0.02$. All times are given in seconds.
			\label{Tab:mp2:1}}}
\end{table}

\section{Concluding remarks and future work}\label{sec:conclusion}
In the work at hand, we have derived a new two-scale reduced scheme for the PG--LOD.\@
Due to the two-stage reduction process and the independence of the resulting ROMs from the size of both the fine-scale
and the corse-scale meshes, our approach is effective even for large-scale problems.
For ease of presentation, we have assumed non-parametric right-hand sides and did not cosider output functionals.
Incorporating both into our approach is straightforward.
Furthermore, for very large coarse meshes, additional intermediate reduction stages could be added to further
reduce the needed computational effort in Stage 2.
Instead of a fixed a priori choice of the Stage~1 tolerance $\varepsilon_1$, the Stage~1 ROMs could be adaptively enriched
during Stage~2 when an insufficient approximation quality of some of the Stage~1 ROMs is detected.
The presented methodology could also be applied to other problem classes and to other multiscale methods.

\section*{Code availability}
All experiments have been implemented in \texttt{Python} using \texttt{gridlod}~\cite{gridlod} for the PG--LOD discretization and \texttt{pyMOR}~\cite{pymor} for the Model Order Reduction.
For the complete source code of all experiments including setup instructions, we refer to~\cite{code}.

{\small
\bibliographystyle{plain}

}

\end{document}